\documentclass[oneside]{amsart}

\usepackage[letterpaper, total={6.5in, 8.5in}]{geometry}
\usepackage{amsmath}
\usepackage{amsfonts}
\usepackage{amssymb}
\usepackage[dvipsnames]{xcolor}
\usepackage{mathrsfs}
\usepackage{enumerate}
\usepackage{amsthm}
\usepackage{xcolor}
\usepackage{color}
\usepackage[pdfdisplaydoctitle=true,
            colorlinks=true,
            urlcolor=blue,
            citecolor=blue,
            linkcolor=blue,
            pdfstartview=FitH,
            pdfpagemode= UseNone,
            bookmarksnumbered=true]{hyperref}
\usepackage{tikz}
\usetikzlibrary{decorations.pathreplacing}
\usepackage{mathrsfs}
\usepackage{pgfplots}
\usepackage[normalem]{ulem}
\usepgfplotslibrary{fillbetween}
\pgfplotsset{compat=1.18}

\definecolor{darkgreen}{rgb}{0,0.5,0}
\definecolor{darkred}{rgb}{0.7, 0, 0}

\makeatletter
\def\@tocline#1#2#3#4#5#6#7{\relax
  \ifnum #1>\c@tocdepth 
  \else
    \par \addpenalty\@secpenalty\addvspace{#2}%
    \begingroup \hyphenpenalty\@M
    \@ifempty{#4}{%
      \@tempdima\csname r@tocindent\number#1\endcsname\relax
    }{%
      \@tempdima#4\relax
    }%
    \parindent\z@ \leftskip#3\relax \advance\leftskip\@tempdima\relax
    \rightskip\@pnumwidth plus4em \parfillskip-\@pnumwidth
    #5\leavevmode\hskip-\@tempdima
      \ifcase #1
       \or\or \hskip 1em \or \hskip 2em \else \hskip 3em \fi%
      #6\nobreak\relax
    \hfill\hbox to\@pnumwidth{\@tocpagenum{#7}}\par
    \nobreak
    \endgroup
  \fi}
\makeatother

\newtheorem{theorem}{Theorem}[section]
\newtheorem{lemma}[theorem]{Lemma}
\newtheorem{proposition}[theorem]{Proposition}
\newtheorem{corollary}[theorem]{Corollary}

\theoremstyle{definition}

\newtheorem{example}[theorem]{Example}

\newtheorem*{theorem*}{Theorem}
\theoremstyle{remark}
\newtheorem{remark}[theorem]{Remark}

\numberwithin{equation}{section}
\newcommand{\R}{\mathbb{R}}
\newcommand{\T}{\mathbb{T}}

\newcommand{\Z}{\mathbb{Z}}
\newcommand{\C}{\mathbb{C}}

\newcommand{\G}{\mathscr{G}}
\newcommand{\g}{\mathfrak{g}}

\newcommand{\symp}{\nabla^\perp}
\DeclareMathOperator{\curl}{curl}

\DeclareMathOperator{\Ad}{Ad}
\DeclareMathOperator{\ad}{ad}
\DeclareMathOperator{\adstar}{\ad^\star\!\!}
\DeclareMathOperator{\Adstar}{\Ad^{\star}\!\!}

\DeclareMathOperator{\diver}{div}

\newcommand{\gflat}{g_{\mathrm{flat}}}
\newcommand{\ghyp}{g_{\mathrm{hyp}}}
\newcommand{\ground}{g_{\mathrm{round}}}

\newcommand{\abs}[1]{\lvert#1\rvert}

\newcommand{\ip}[1]{\langle #1 \rangle}

\newcommand{\ce}{\epsilon}
\newcommand{\proj}{\tau}

\newcommand{\inert}{\mathcal{A}} 
\newcommand{\vf}{\mathfrak{X}} 
\newcommand{\df}{\mathfrak{X}_{\mu}} 
\newcommand{\hm}{\mathcal{H}} 
\newcommand{\ff}{\mathfrak{X}_{\mu,\text{ex}}} 
\newcommand{\ef}{\mathfrak{q}} 
\newcommand{\kev}{\lambda} 
\newcommand{\iev}{\alpha} 
\newcommand{\aev}{\zeta} 
\newcommand{\cs}{U_{\mathfrak{C}}} 
\newcommand{\rs}{U_{\mathfrak{R}}} 
\newcommand{\is}{U_{\mathfrak{I}}} 
\newcommand{\cecs}{\mathcal{Q}} 
\newcommand{\freq}{\varpi} 

\newcommand{\assleg}{Q}
\newcommand{\curv}{\kappa}
\newcommand{\scurv}{s_{\curv}}

\newcommand{\eulerclass}{k}
\newcommand{\stream}{f}
\newcommand{\fnone}{\varphi}
\newcommand{\fntwo}{\varsigma}
\newcommand{\fnthree}{a}
\newcommand{\fnfour}{b}
\newcommand{\fnfive}{f}
\newcommand{\radiusone}{{\varepsilon_1}}
\newcommand{\radiustwo}{{\varepsilon_2}}

\author[1]{Patrick Heslin}
\author[2]{Stephen C. Preston}

\address{P. Heslin: Max Planck Institute for Mathematics in the Sciences, 04103, Leipzig, Germany}
\email{patrick.heslin@mis.mpg.de}
\address{S.C. Preston: Brooklyn College and CUNY Graduate Center, USA}
\email{stephen.preston@brooklyn.cuny.edu}

\begin{document}

\title[Exact non-stationary solutions of the Euler equations]{Exact non-stationary solutions of the Euler equations \\ in two and three dimensions}


\begin{abstract}
We develop, via Arnold’s geometric framework, a mechanism for constructing explicit, smooth, global-in-time, and typically non-stationary solutions of the incompressible Euler equations. The approach introduces a notion of generalized Coriolis force, whose spectrum underlies the construction of these solutions. We recover classical exact solutions such as Kelvin and Rossby-Haurwitz waves, while also producing new explicit examples on curved surfaces and three-dimensional manifolds including the round three-sphere. Furthermore, we obtain a complete classification in two dimensions and a partial classification in three dimensions of the Riemannian manifolds that admit such solutions. The method is in fact formulated in the general Euler-Arnold setting and yields a simple criterion for non-stationarity.
\end{abstract}

\maketitle

\tableofcontents

\enlargethispage{2\baselineskip}

\section{Introduction}\label{introduction}

Introduced by Euler in 1757, the equations of ideal hydrodynamics
\begin{equation}\label{euler}
    \begin{split}
        &\partial_t u + \nabla_u u = -\nabla p, \\
        &\diver u = 0,
    \end{split}
\end{equation}
describe the motion of an incompressible, inviscid fluid. The first rigorous results on local existence and uniqueness for \eqref{euler} date back to the early twentieth century, obtained independently by G{\"u}nther and Lichtenstein. Wolibner later demonstrated that solutions for two-dimensional fluids extend globally in time. Subsequent advances, notably by Yudovich, Kato, and others, sharpened these results and clarified the role of regularity and vorticity structure. A detailed historical account, along with further developments and references, can be found in the book of Majda and Bertozzi \cite{majda2002vorticity} and the survey of Drivas and Elgindi \cite{drivas2023singularity}.

Arnold’s 1966 work \cite{arnold2013differential} recast ideal fluid motion in geometric terms. In this formulation, the Euler equations coincide with the reduced geodesic equation on the group of volume-preserving diffeomorphisms equipped with a right-invariant $L^2$ metric arising from kinetic energy. Ebin and Marsden \cite{ebin1970groups} proved that the corresponding geodesic flow in the $H^s$ Sobolev setting is governed by an ordinary differential equation and obtained smoothness of the associated Riemannian exponential map, which may be identified with the Lagrangian flow map of \eqref{euler}. This geometric formulation has since played a central role in the study of stability, curvature, and qualitative features of fluid motion. An extensive treatment of this approach can be found in the monograph of Arnold and Khesin \cite{arnold2021topological}.

Despite the extensive literature on the Euler equations, comparatively few classes of explicit, smooth, non-stationary solutions are known, particularly in three dimensions. Classical examples include Kelvin waves, Rossby-Haurwitz waves on the sphere, and certain rotating or shearing flows; see, for example, \cite{craik1986evolution, lamb1993hydrodynamics, skiba2017mathematical}. More recent examples along the same lines can be found in Vi{\' u}dez~\cite{viudez2022exact}. Many other explicit constructions appearing in the literature either possess infinite kinetic energy, rely on singular vorticity distributions, or arise as solutions to reduced or symmetry-constrained models \cite{aref2007point, sarria2013blow}. A detailed discussion of such examples can be found in \cite{majda2002vorticity} and the references therein.

\enlargethispage{2\baselineskip}

The central result of this paper utilizes Arnold’s geometric framework to provide a general mechanism for constructing a class of explicit, smooth, global-in-time and typically non-stationary solutions to the Euler equations \eqref{euler} and, more broadly, to Euler-Arnold equations on Lie groups. The underlying mechanism can be understood from two complementary viewpoints.

The first viewpoint is perturbative, though no smallness assumption is imposed on the size of the perturbation. One begins with a Killing field $X$ whose flow, by definition, acts by isometries and hence produces a steady Euler flow. Introducing a perturbation of this steady solution and formally expanding the resulting solution of \eqref{euler} leads to an infinite system of coupled ordinary differential equations governing the evolution of the expansion coefficients. We demonstrate that if the second-order term vanishes, all higher-order terms vanish. Hence the solution of the linearized Euler equation is also a solution of the full Euler equation.

A particularly simple class of solutions to this system arises when the perturbation is chosen from a complex simultaneous eigenfield of the inertia operator\footnote{For ideal hydrodynamics, this is the Hodge Laplacian when the fluid is two-dimensional and the curl operator when the fluid is three-dimensional.} and the Lie algebra coadjoint operator associated with $X$. In this situation the nonlinear interactions close on a finite-dimensional subspace, producing explicit time-dependent solutions. Importantly, this construction yields exact solutions of the fully nonlinear Euler equations \eqref{euler}, rather than solutions of a linearized model.

A second more geometric viewpoint proceeds at the level of Lagrangian flows. One begins with the flow generated by an arbitrary time-dependent divergence-free vector field $V(t)$ and precomposes it with the flow generated by a Killing field $X$. Requiring that the resulting flow satisfies the Euler equations leads to an evolution equation for $V(t)$. The structure of this equation reveals that the spectrum of the coadjoint operator associated with $X$ naturally produces a particularly simple class of its solutions. Once a suitable solution $V(t)$ is obtained, precomposing again with the isometric flow generated by $X$ produces a corresponding solution of the Euler equations. Moreover, the underlying Lie algebraic structure provides a simple criterion for determining when the resulting Euler flow is genuinely non-stationary.

When applied to two-dimensional ideal hydrodynamics our construction recovers several well-known solutions, including Kelvin waves on flat domains and Rossby-Haurwitz waves on the round sphere\footnote{Note that some of these have been independently derived as Euler solutions by other authors~\cite{azencot2015discrete, chern2024force}}. In addition, we obtain explicit solutions on curved surfaces such as the hyperbolic disk. In three dimensions we produce new explicit solutions of the Euler equations on manifolds including the round three-sphere and various circle and torus bundles with nontrivial geometry. On the three-sphere, these solutions may be viewed as higher-dimensional analogues of Rossby-Haurwitz waves: they are periodic in time and occur in families of multiplicity greater than one.

Although our primary focus is on explicit constructions, the geometric nature of the solutions suggests that they may exhibit interesting stability properties. We expect that the methods developed here will be useful in future investigations in this direction from the Euler-Arnold viewpoint.

In all of our examples the chosen Killing field has the additional property that its image under the inertia operator is also Killing. For example, on a flat cylinder, the curl of the rotational vector field generates vertical translations, which are themselves isometries. All explicit examples presented in this paper fall into this category. We therefore undertake a systematic study of manifolds admitting such Killing fields. In two dimensions, we obtain a complete classification, while in three dimensions we give a partial classification.

We emphasize that although the motivating examples come from ideal hydrodynamics, the construction itself applies to a wide class of Euler-Arnold equations on Lie groups endowed with right-invariant metrics. The construction further admits a natural interpretation in terms of a generalized Coriolis force, which we outline. This perspective clarifies the role of classical rotating-wave solutions and places them within a broader geometric framework.

The paper is organized as follows. Section~\ref{ideal hydrodynamics} introduces the general mechanism for constructing solutions in the context of ideal hydrodynamics and formulates sufficient conditions under which the underlying manifold admits an abundance of such solutions. Sections~\ref{2D} and~\ref{3D} provide a full and partial classification, respectively, of two- and three-dimensional manifolds satisfying these conditions, together with explicit examples. Section~\ref{general framework} develops the construction in the general Euler-Arnold setting. Section~\ref{future} outlines directions for future work. Lastly, Appendix \ref{geometric lemmas} contains proofs of a Riemannian geometric nature which are deferred for readability.

\enlargethispage{4\baselineskip}

\subsection*{Acknowledgments}
The authors thank Abhijit Champanerkar, Albert Chern, Theodore Drivas, Boris Khesin, Gerard Misio{\l}ek, Daniel Peralta-Salas and Cornelia Vizman for helpful discussions and comments.

\section{Ideal Hydrodynamics}\label{ideal hydrodynamics}
Let $(M,g)$ be a two- or three-dimensional compact Riemannian manifold, possibly with boundary. 
Denote by $\df(M)$ the space of smooth divergence-free vector fields on $M$ that are tangent to the boundary. 
Within $\df(M)$ there is the finite-dimensional subspace of harmonic fields $\hm(M)$, consisting of divergence-free and curl-free fields tangent to the boundary. By the Hodge decomposition (cf.~\cite{ebin1970groups}) we have the splitting
\begin{equation*}
    \df(M) = \ff(M) \oplus_{L^2} \hm(M),
\end{equation*}
where $\ff(M)$ denotes the $L^2$-orthogonal complement of $\hm(M)$ in $\df(M)$. We recall the usual musical isomorphisms $\flat$ and $\sharp$, the Hodge star operator $\star$, the exterior derivative $d$, and its formal adjoint $\delta$.

Note that, in two dimensions, every element of $\ff(M)$ can be written as a skew-gradient
\begin{equation*}
    v = \symp \psi = (\star d\psi)^{\sharp},
\end{equation*}
and we refer to $\psi$ as a stream function for $v$.

We define the inertia operator $\inert$ by
\begin{equation}\label{inertia operator}
    \inert =
    \begin{cases}
        \Delta & \dim(M)=2, \\
        \curl & \dim(M)=3,
    \end{cases}
\end{equation}
where $\Delta$ is the positive-definite Hodge Laplacian
\begin{equation}\label{hodge laplacian}
    \Delta u = (d\delta + \delta d)(u^\flat)^{\sharp},
\end{equation}
and the curl operator is defined by
\begin{equation}\label{curl}
    \curl u = (\star (du^\flat))^{\sharp}.
\end{equation}

\begin{lemma}\label{inertia basis}
    Let $(M,g)$ be a two- or three-dimensional compact Riemannian manifold, possibly with boundary. The inertia operator $\inert$ given in \eqref{inertia operator} is invertible on $\ff(M)$, and its inverse $\inert^{-1}$ is compact and self-adjoint in $L^2$. Consequently, $\inert$ has a discrete spectrum $\{\iev_k\}$ with finite-dimensional eigenspaces, and the corresponding eigenfields form an $L^2$-orthonormal basis of $\ff(M)$.
\end{lemma}

\begin{proof}
In two dimensions, for each $v \in \ff$, we have $v = \symp \psi$, for some stream function $\psi$ vanishing on all boundary components. Hence, we have
$\inert \symp \psi = \symp \Delta \psi,$
so that $\inert^{-1}(\symp \psi) = \symp \Delta^{-1}\psi$, where $\Delta^{-1}$ denotes the inverse Laplacian with Dirichlet boundary conditions on the stream function. The $L^2$ norm on skew-gradients coincides with the Dirichlet norm on the stream functions, and $\Delta^{-1}$ is compact and self-adjoint in this inner product, cf. \cite{evans2010partial}. 

In three dimensions, the kernel of $\inert$ is the space of vector fields tangent to the boundary for which the divergence and curl both vanish, which is precisely $\hm$. The result of Yoshida-Giga~\cite{yoshidagiga} then shows that $\inert$ is invertible on $\ff$ and self-adjoint in the case where $M\subset \mathbb{R}^3$, and the same argument extends to any compact Riemannian $3$-manifold with boundary.
\end{proof}

Next, let $u_0 \in \ff(M)$ and consider the operator given by
\begin{equation}\label{coadjoint operator}
    K_{u_0} : \ff(M) \rightarrow \ff(M) \ ; \ v \mapsto \inert^{-1}[v, \inert u_0]
\end{equation}
where for $v,w \in \ff(M)$ the bracket $[v,w] = vw-wv$ denotes the standard commutator of vector fields. In accordance with the framework outlined in Section~\ref{general framework}, we will refer to the operators $\ad_{u_0} : v \mapsto -[u_0, v]$ and $v\mapsto K_{u_0}v$ as the adjoint and coadjoint operators associated to $u_0$ respectively. The operator $K_{u_0}$ is compact if the fluid is two-dimensional, while it is generally only bounded if the fluid is three-dimensional; see \cite{ebin2006singularities} for further discussion and consequences.

With this notation, the Euler equations \eqref{euler} can be written in the form
\begin{equation}\label{inertia euler}
    \partial_t \inert u + [u,\inert u] = 0, 
    \qquad \diver u = 0, 
    \qquad u \parallel \partial M.
\end{equation}
This equation uniquely determines $u$ up to harmonic fields.

We are now ready to state our main theorem in the context of ideal hydrodynamics.

\begin{theorem}\label{simultaneous eigenvalue to euler solution}
    Let $(M,g)$ be a Riemannian manifold, possibly with boundary, and let $u_0 \in \ff(M)$ be a steady solution of the Euler equations \eqref{euler}, i.e.,\ $[u_0,\inert u_0]=0$.
    If there exists a complex vector field $\ef=v+iw\in \C\otimes\ff(M)$, with $v,w\in\ff(M)$, and real constants $\kev,\iev,\aev$ such that
    \begin{equation*}
        K_{u_0} \ef = -i \kev \ef, \qquad \inert \ef = \iev \ef, \qquad [u_0, \ef] = i \aev \ef,
    \end{equation*}
    then, defining $\freq=\kev-\aev$, the curve
    \begin{equation*}
        \cs(t) = u_0 + e^{i\freq t}\ef
    \end{equation*}
    is a complex-valued solution of the Euler equations \eqref{euler}.
    
    Furthermore, the real and imaginary parts
    \begin{equation*}
        \rs(t) = u_0 + \mathrm{Re}\big(e^{i\freq t}\ef\big) = u_0 + \cos(\freq t)v -  \sin(\freq t)w
    \end{equation*}
    and
    \begin{equation*}
        \is(t) = \mathrm{Im}\big(e^{i\freq t}\ef\big) =  \sin(\freq t)v + \cos(\freq t)w
    \end{equation*}
    are real solutions to the Euler equation \eqref{euler} and its linearization at $\rs$ respectively.
    
    Finally, the real-valued solution $\rs$ is stationary if and only if $\kev=\aev$.
\end{theorem}

\begin{remark}
    In fact, for each field $\ef$ we have a two-dimensional space of such solutions which corresponds to the complex span of $\ef$. Replacing $\ef$ with $\rho e^{i\sigma}\ef$ for $\rho \geq 0$ and $\sigma \in \R$ yields the entire set of solutions generated by $\ef$
    \begin{align*}
        \cs(t) &= u_0 + \rho e^{i( \sigma + \freq t)}\ef \\
        \rs(t) &= u_0 + \mathrm{Re}\Big(\rho e^{i ( \sigma + \freq t)}\ef\Big) = u_0 + \rho\cos(\sigma + \freq t)v - \rho \sin(\sigma + \freq t)w \\
        \is(t)&= \mathrm{Im}\Big(\rho e^{i ( \sigma + \freq t)}\ef\Big) = \rho \sin(\sigma + \freq t)v + \rho\cos(\sigma + \freq t)w.
    \end{align*}
    The parameter $\rho$ governs the size of the perturbation, relative to the fixed steady flow $u_0$, while $\sigma$ gives a phase. However, we will continue in this section, and indeed largely in Sections \ref{2D} and \ref{3D}, with $\rho=1$ and $\sigma=0$ for simplicity.
\end{remark}

\begin{proof}[Proof of Theorem~\ref{simultaneous eigenvalue to euler solution}]
    Computing directly, we have
    \begin{equation*}
        \partial_t \inert \cs = \partial_t \left( \inert u_0 +  \iev e^{i \freq t}\ef \right) = i \iev \freq e^{i \freq t}\ef
    \end{equation*}
    along with
    \begin{align*}
         [\cs, \inert \cs] &=  \left[ u_0 +  e^{i \freq t}\ef, \inert u_0 +  \iev e^{i \freq t}\ef\right] \\
        &= \left[ u_0, \inert u_0 \right] +  e^{i \freq t} \Big(\iev\left[u_0, \ef\right] + \left[\ef, \inert u_0\right] \Big) + e^{2i\freq t}[\ef,\ef] \\
        &=  e^{i\freq t} \Big( i \iev \aev \ef + \inert K_{u_0}\ef \Big) \\
        &= i \iev (\aev - \kev)e^{i\freq t}\ef \\
        &= -i \iev \freq e^{i\freq t}\ef
    \end{align*}
    and hence $\cs$ is a complex-valued solution of the Euler equations \eqref{euler}.

    Decomposing $U = \rs + i \is$ in $\partial_t \inert U + [U, \inert U] = 0$ yields
    \begin{equation*}
        \partial_t \inert \rs + [\rs, \inert \rs] - [\is, \inert \is] = 0, \qquad \partial_t \inert \is + [\rs, \inert \is] + [\is, \inert \rs] = 0.
    \end{equation*}
    Notice now that $\inert \ef  = \iev \ef$ implies both $\inert v = \iev v$ and $\inert w = \iev w$ which gives us $\inert \is = \iev \is$. Hence $[\is, \inert \is] = 0$ and we see that $\rs$ is a real solution of the Euler equations \eqref{euler}. Meanwhile the second equation is precisely the statement that $\is$ satisfies the linearized Euler equations at $\rs$.

    For the stationarity condition, we can compute directly that
    \begin{equation*}
        \partial_t \rs =  (\kev - \aev) \is
    \end{equation*}
    which completes the proof.
\end{proof}

\begin{remark}\label{euler trivial solutions}
    As will be explained in more detail in Section \ref{general framework}, the flow of our solution $\rs$ and the flow of the field $V(t) = \mathrm{Re}\left(e^{i \kev t}\ef\right)$ are related by the equation
    \begin{equation*}
        \gamma_{\rs} = \gamma_{u_0} \circ \gamma_V.
    \end{equation*}
    Notice that it may occur that $\rs$ is a non-stationary solution of \eqref{euler}, while also having that $V$ is a stationary field. In this situation $\rs$ can be interpreted as a trivial solution in a moving frame generated by $u_0$; like coffee being stirred on a moving train. However, we see that
    \begin{equation*}
        \partial_t V = 0 \iff \kev = 0. 
    \end{equation*}
    Accordingly, eigenfields $\ef$ with $\kev = 0$ produce solutions that are trivial in this moving-frame sense, whereas $\kev \neq 0$ yields solutions $\rs$ that are genuinely non-trivial. This corresponds to the importance of vorticity twisting.
\end{remark}

We now turn to the abundance of such solutions. A vector field $X \in \ff(M)$ is said to be a Killing field if its flow $\gamma_X$ acts by isometries. Equivalently, for all $u,v\in \ff(M)$ we have
\begin{equation}\label{Killing condition}
    g(\nabla_uX, v) + g(u, \nabla_v X) = 0.
\end{equation}
As we will demonstrate in the appendix, any such field is automatically a steady solution of \eqref{euler}, cf.~\cite{misiolek1993stability}.

\begin{remark}
    If the vector field $u_0$ in Theorem \ref{simultaneous eigenvalue to euler solution} is in fact Killing and the complex field $\ef$ is such that $\aev = 0$, then the corresponding solution $\rs$ is axi-symmetric with respect to $u_0$.

    Furthermore, for simplicity in the spectral theory, we have restricted to $\ef \in \C \otimes \ff(M)$, the complexified $L^2$-orthogonal complement of harmonic fields within divergence-free fields; hence we we will not have $\alpha = 0$ for non-trivial $\ef$. However we note that there exist exact solutions of \eqref{euler} such that $\inert u$ is steady, but $u$ has an oscillating harmonic component, cf. Yin et al.~\cite{chern2023fluid}. Such dynamics may be interesting in this context.
\end{remark}

\enlargethispage{2\baselineskip}

Our central ingredient is the following.

\begin{proposition}\label{abundance lemma}
    Let $(M,g)$ be a compact Riemannian manifold, possibly with boundary, and $X \in \ff(M)$ be a Killing field such that $\inert X$ is also Killing. Then the operators
    \begin{equation*}
        K_X : \ef \mapsto K_X\ef, \quad \inert : \ef \mapsto \inert \ef, \quad \ad_X: \ef\mapsto -[X,\ef]    
    \end{equation*}
    admit a discrete simultaneous eigenbasis of $\C \otimes\ff(M)$. 
\end{proposition}

In particular, if both $X$ and $\inert X$ are Killing, then there exists an infinite-dimensional family of $\ef \in \C \otimes \ff(M)$ satisfying the hypotheses of Theorem~\ref{simultaneous eigenvalue to euler solution}, with $u_0=X$, and hence an infinite collection of exact, non-stationary solutions of \eqref{euler}. We present the proof in Appendix \ref{geometric lemmas}.

\begin{theorem}\label{loads of solutions}
    Let $(M,g)$ be a compact Riemannian manifold, possibly with boundary and denote by $\ff(M)$ those divergence-free vector fields which are tangent to the boundary of $M$ and $L^2$-orthogonal to harmonic fields.
    
    If there exists a Killing field $X \in \ff(M)$ such that $\inert X$ is also Killing, where
    \begin{equation*}
        \inert = \begin{cases}
        \Delta & \dim(M)=2, \\
        \curl & \dim(M)=3,
    \end{cases}
    \end{equation*}
    then there exists infinitely many $v,w \in \ff(M)$ with constants $\kev, \iev$ and $\aev$ such that
    \begin{align*}
        \inert^{-1}[v,\inert X] &= \kev w ,
        &\quad \inert v &= \iev v, 
        &\quad [X, v] &= -\aev w, \\
        \inert^{-1}[w,\inert X] &= - \kev v, 
        &\quad \inert w &= \iev w,
        &\quad [X, w] &= \aev v.
    \end{align*}
    Consequently, for all $\rho \geq 0$ and $\sigma\in \R$, we have a real solution of the Euler equations \eqref{euler} given by
    \begin{equation}
        \rs(t) = X + \rho\cos(\sigma + \freq t)v - \rho\sin(\sigma + \freq t)w
    \end{equation}
    where $\freq = \kev - \aev$.
\end{theorem}

\begin{proof}
    From Proposition \ref{abundance lemma}, we have an entire basis of complex fields $\ef \in \C \otimes \ff$ satisfying the assumptions of Theorem \ref{simultaneous eigenvalue to euler solution} with $u_0 = X$. Decomposing each as $\ef = v + iw$ with $v,w \in \ff$ and recalling the definition of the operator $K_X$ from \eqref{coadjoint operator} completes the proof. \\
\end{proof}

\begin{remark}
    Note that locally we can construct polar coordinates such that $X=\partial_\theta$ which in turn induces a traveling wave type structure on our solutions in the corresponding chart; that is, the dependence on $\theta$ is in terms of $(\theta-t)$. It is conceivable however that no such global coordinate $\theta$ exists, depending on whether the orbits of $X$ are closed and of the same length.
\end{remark} 

In Sections~\ref{2D} and~\ref{3D} we classify, in two and three dimensions respectively, the manifolds admitting Killing fields $X$ such that $\inert X$ is also Killing. In each case we provide explicit examples of the resulting solutions.

\section{Two-Dimensional Fluids}\label{2D}
For this section we have $\inert=\Delta$, the positive-definite Hodge Laplacian on vector fields, cf. \eqref{hodge laplacian}. We begin by showing, via the Bochner technique, that any surface satisfying the conditions of Theorem \ref{simultaneous eigenvalue to euler solution} must have constant curvature. The proof can also be carried out in local coordinates, but the technique of Bochner yields a global result more readily. We refer to Petersen~\cite{petersen2006riemannian} for the notation.

\begin{theorem}\label{surfacetheorem}
Let $(M,g)$ be a two-dimensional Riemannian manifold, possibly with boundary $\partial M$. If $M$ admits a Killing field $X$ whose Hodge Laplacian $\Delta X$ is also Killing, then the sectional curvature of $M$ is constant.
\end{theorem}

\begin{proof}
  The Weitzenbock formula says on any Riemannian manifold that
        \begin{equation*}
            \Delta(v^\flat) = \nabla^* \nabla v^\flat + \mathrm{Ric}(v^\flat), \quad \text{for any } v \in \mathfrak{X}(M).
        \end{equation*}
        For a Killing field $X$ we have\footnote{See \cite[Proposition 8.1.3.]{petersen2006riemannian}.}
        \begin{equation*}
            \nabla^2_{v,w} X = - R_{X,v}w, \quad \text{for all } v, w \in \mathfrak{X}(M)
        \end{equation*} 
        and, as a consequence,
        \begin{equation*}
            \Delta X = 2\mathrm{Ric}(X).
        \end{equation*} 
        If now $M$ is a surface, then $\mathrm{Ric} = \kappa g$, where $\kappa$ is the sectional curvature. So our requirement becomes that $\kappa X$ is also Killing. Computing
        \begin{align*}
        \mathcal{L}_{\kappa X}g(v,w) &= g(\nabla_v(\kappa X), w) + g(v, \nabla_{w}(\kappa X)) \\
            &= \kappa \big( g(\nabla_v X, w) + g(v, \nabla_w X) \big) + g(v(\kappa)X, w) + g(v, w(\kappa)X) \\
            & = g(v(\kappa)X, w) + g(v, w(\kappa)X).
        \end{align*}
        This must vanish for any $v,w$. Taking $v=\nabla \kappa$ and $w=X$, we get
        \begin{equation*}
            \abs{\nabla \kappa}_g^2 \abs{X}_g^2 + \abs{X(\kappa)}_g^2 = 0.
        \end{equation*}
        Hence, $\kappa$ is constant.
\end{proof}
Note that when the fluid domain is a compact surface of constant sectional curvature $\kappa$, we can express the metric locally in polar coordinates
    \begin{equation*}\label{constant curvature metric}
        g = dr^2 + s_{\kappa}(r)d\theta^2.
    \end{equation*}
where
\begin{equation}\label{generalizedtrig}
    \scurv(r) = \begin{cases} 
       \sin{r} & \curv=1, \\
       r & \curv=0, \\
       \sinh{r} & \curv=-1,
    \end{cases}
\end{equation}
and $X = \partial_\theta$ is Killing with $\Delta X = 2 \kappa X$.

The following theorem outlines the general scheme for applying Theorem \eqref{simultaneous eigenvalue to euler solution} in this setting.

\begin{theorem}\label{two-dimensional examples}
    Let $(M,g)$ be a surface of constant sectional curvature $\kappa$. The eigenfunctions of the Laplacian are of the form
    \begin{equation*}
        \psi_{n,m}(r,\theta)=e^{in\theta}F_{n,m}(r), \qquad \Delta\psi_{n,m}=\iev_{n,m}\psi_{n,m}, \quad n \in\Z, m \in \Z_{\geq0}
    \end{equation*}
    and, defining $\freq_{n,m} = n\left(\frac{2\kappa}{\iev_{n,m}}-1\right)$, the vector field
    \begin{equation}\label{constant curvature surface solutions}
        \rs(t,r,\theta) = \partial_\theta + \frac{F_{n,m}^{\prime}(r)}{s_\kappa(r)}\cos\left(n\theta + \freq_{n,m}t\right)\partial_\theta + \frac{nF_{n,m}(r)}{s_\kappa(r)}\sin\left(n\theta + \freq_{n,m}t\right)\partial_r
    \end{equation}
    solves the Euler equations \eqref{euler}.

    Furthermore $\rs$ is stationary if and only if $\freq_{n,m}=0$.
\end{theorem}

\begin{proof}
    Taking the skew-gradient of the Laplace-eigenfunctions, we acquire divergence-free fields
    \begin{equation*}
    {\ef}_{n,m} = \symp\psi_{n,m} = \frac{1}{s_{\kappa}(r)}\big(-\partial_\theta \psi_{n,m} \partial_r + \partial_r \psi_{n,m}\partial_\theta\big) = \frac{e^{in\theta}}{s_{\kappa}(r)} \big(-inF_{n,m}(r)\partial_r + F_{n,m}^{\prime}(r)\partial_\theta\big)
    \end{equation*}
    satisfying
    \begin{equation*}
        \Delta \ef_{n,m} = \iev_{n,m} \ef_{n,m}.
    \end{equation*}
    These are eigenfields of our adjoint 
    \begin{equation*}
        [\partial_\theta, {\ef}_{n,m}] = \symp \partial_\theta\psi_{n,m} = in {\ef}_{n,m}
    \end{equation*}
    and, from \eqref{coadjoint operator}, our coadjoint operator
    \begin{equation*}
        K_{ \partial_\theta}{\ef}_{n,m} = \Delta^{-1}[{\ef}_{n,m}, \Delta \partial_{\theta}] = -i\frac{2 \kappa n}{\iev_{n,m}}{\ef}_{n,m}.
    \end{equation*}
    Hence, in the context of Theorem \ref{simultaneous eigenvalue to euler solution}, we have $\kev_{n,m} = \frac{2 \kappa n}{\alpha_{n,m}}$ and $\aev_{n,m} = n$.

    Letting $\freq_{n,m} = \kev_{n,m} - \aev_{n,m} = n\left(\frac{2\kappa}{\iev_{n,m}}-1\right)$, the corresponding complex-valued solution of \eqref{euler} is given by
    \begin{equation*}
        \cs(t) = \partial_\theta + e^{i\freq_{n,m}t}\ef_{n,m}.
    \end{equation*}
    Taking the real part yields \eqref{constant curvature surface solutions}.
\end{proof}
We now consider some canonical compact constant curvature surfaces.
\begin{example}[Kelvin Waves on the Flat Disk]
Let $(D^2,\gflat)$ be the flat two-dimensional unit disk and consider the Killing field $X=\partial_\theta$ which generates rotation around the origin. Here we have $\kappa=0$, and hence $\Delta \partial_\theta = 0$. Consequently, the only admissible value of $\kev$ in Theorem~\ref{simultaneous eigenvalue to euler solution} is $\kev=0$. Hence, the solutions we generate in this setting are trivial in the sense outlined in Remark \ref{euler trivial solutions}.

On the flat unit disk, $s_\kappa(r)=r$ and the eigenfunctions of the Laplacian are given by
\begin{equation*}
\psi_{n,m}(r,\theta)=J_n(\beta_{n,m}r)e^{in\theta}, \qquad \Delta\psi_{n,m}=\beta_{n,m}^2\psi_{n,m}, \quad n\in\Z, m \in \Z_{\geq1}.
\end{equation*}
where $J_n$ is the $n^{\text{th}}$ Bessel function and $\beta_{n,m}$ its $m^{\text{th}}$ zero. Letting $\freq_{n,m} = -n$, the corresponding solutions \eqref{constant curvature surface solutions} are precisely the classical Kelvin waves on the flat unit disk.

In Section \ref{3D} we present a three-dimensional analogue of the above where the helical symmetry creates a twisting effect which generates nontrivial flows, cf. Example \ref{cequalszero}.
\end{example}

On the round sphere, the presence of positive curvature alters the action of the Laplacian on Killing fields, leading to nontrivial eigenvalues $\kev$.

\enlargethispage{3\baselineskip}

\begin{example}[Rossby-Haurwitz Waves on the Round Two-Sphere]\label{2D RH} Let $(S^2,\ground)$ be the round two-sphere and consider the Killing field $X=\partial_\theta$ which generates equatorial rotation. Here we have $\Delta \partial_\theta = 2\partial_\theta$ and, replacing $r$ with the more natural $\phi \in [0,\pi]$ denoting the latitudinal parameter, $s_\kappa(\phi) = \sin(\phi)$. The spherical harmonics are given by
\begin{equation*}
\psi_{n,m}(\theta, \phi)=P_{n,m}(\cos\phi)e^{in\theta}, \qquad \Delta\psi_{n,m}=m(m+1)\psi_{n,m}, \quad m\in\Z_{\ge0}, \quad n\in\Z, \quad \abs{n}\le m
,
\end{equation*}
where $P_{n,m}$ is the associated Legendre polynomial with the standard normalization. Letting $\freq_{n,m} = n\left(\frac{2}{m(m+1)}-1\right)$, the corresponding solutions \eqref{constant curvature surface solutions} are precisely the Rossby-Haurwitz waves on the round two-sphere. Furthermore, $\rs$ is stationary if and only if $\frac{2n}{m(m+1)} = n$, that is, $n=0$ or $m = 1$.

The simplest non-stationary example is then $n=1$ and $m=2$. Here, we get that $P_{1,2}(\cos{\phi})$ is a multiple of $\sin{\phi} \cos{\phi}$ and $\freq_{1,2} = -\tfrac{2}{3}$. The corresponding solution from \eqref{constant curvature surface solutions} is given by
\begin{equation}\label{explicit2drossby}
\rs(t,\phi,\theta) = \partial_\theta + \frac{\cos(2\phi)}{\sin (\phi)}\cos\left(\theta - \frac{2}{3}t\right)\partial_\theta + \cos(\phi)\sin\left(\theta - \frac{2}{3}t\right)\partial_\phi
\end{equation}
which can be expressed in Cartesian coordinates as the restriction to the sphere of
\begin{equation}\label{explicit 2D RHW}
    \rs(t) = e_1 - \left(xe_1 + ze_2\right)\cos\left(\frac{2t}{3}\right) - \left(ye_1 + ze_3\right)\sin\left(\frac{2t}{3}\right)
\end{equation}
where
\begin{equation*}
    e_1 = -y \partial_x + x \partial_y, \quad e_2 = -z\partial_y + y\partial_z, \quad e_3 = -x\partial_z +z\partial_x.
\end{equation*}
In Section \ref{3D} we present a three-dimensional analogue of these solutions on $S^3$, cf. \eqref{explicit 3D RHW}.
\end{example}
Finally, we consider the negative curvature case. Note that by Bochner’s theorem (cf. Petersen~\cite[Chapter 7]{petersen2006riemannian}), a closed manifold with negative Ricci curvature admits no nontrivial Killing fields. Hence, our examples in this setting are restricted to surfaces with boundary. We will focus on the hyperbolic analogue of the flat disk. In contrast to both the flat and spherical cases, negative curvature reverses the sign of the Laplacian on rotational Killing fields, yielding a distinct family of Kelvin-type waves on compact domains in hyperbolic space. 

\begin{example}[Kelvin Waves on a Compact Disk in Hyperbolic Space]
Let $(\mathbb{H}^2,\ghyp)$ be standard hyperbolic space and consider a closed geodesic ball $D^2$ given in polar coordinates by $\{(r,\theta) \ \big\vert \ 0 \leq r \leq 1\}$ with $\ghyp=dr^2 + \sinh^2r d\theta^2$. The vector field $X=\partial_\theta$ is then Killing and $\Delta \partial_\theta = - 2\partial_\theta$. The eigenfunctions of the Laplacian on $(D^2, \ghyp)$ are given by
\begin{equation*} 
\psi_{n,m}(r, \theta)=\assleg_{n,m}(\cosh{r})
e^{in\theta},  \qquad
\Delta\psi_{n,m} =\iev_{n,m}\psi_{n,m}, \quad \text{with } \iev_{n,m}= \tfrac{1}{4} + \beta_{n,m}^2, \quad n\in\Z, m \in \Z_{\geq1},
\end{equation*}
where $\assleg_{n,m}:=Q^{\abs{n}}_{-\tfrac{1}{2}+i\beta_{n,m}}$ is the associated Legendre function\footnote{For details on these eigenfunctions, see \cite{terras2012harmonic}.} with $\beta_{n,m}$ such that $\assleg_{n,m}(\cosh{1})=0$ and the branch cut chosen so that the function is real on $[1,\infty)$. Letting $\freq_{n,m} = n\left(\frac{2}{\iev_{n,m}}-1\right)$, the corresponding solutions \eqref{constant curvature surface solutions} are a hyperbolic analogue of Rossby-Haurwitz waves.

Lastly, $\rs$ is stationary if and only if $\frac{2n}{\iev_{n,m}} = n$, that is, $n=0$ or $\iev_{n,m} = 2$. \\
\end{example}

\section{Three-Dimensional Fluids}\label{3D}
We now consider three-dimensional manifolds, for which the inertia operator is $\inert=\curl$, cf. \eqref{curl}. To construct examples as in Theorem~\ref{simultaneous eigenvalue to euler solution}, we must understand when a manifold admits a Killing field $X$ whose vorticity $\curl X$ is also Killing. As we will see, the analysis naturally splits into two cases depending on whether $\abs{X}_g$ is constant. Independently of the Killing assumption, Arnold’s structure theorem for steady Euler flows \cite[Section II.1]{arnold2021topological} implies that away from singular sets either $\curl X = fX$ for some function $f$ or the flow lines lie on invariant two-dimensional tori. When both $X$ and $\curl X$ are Killing, this dichotomy becomes geometrically rigid. The following lemma provides the key rigidity input.

\begin{proposition}\label{beltramilemma}
    Let $(M,g)$ be a three-dimensional compact connected Riemannian manifold, possibly with boundary. If $X$ and $\curl{X}$ are both Killing fields, then $g(X,\curl X)$ is constant. Consequently, if $\curl X = fX$, then both $\abs{X}_g$ and $f$ are constant.
\end{proposition}
While the proof of the above can be carried out locally in coordinates, a global formulation requires intrinsic, coordinate-free computations involving the Levi-Civita connection. We defer this to the Appendix.

Hence, as mentioned above, Proposition~\ref{beltramilemma} reduces our analysis to two geometrically distinct situations.

If $\curl X$ is everywhere parallel to $X$, then $\curl X=\alpha X$ for some constant $\alpha$, and $X$ has constant length. This is the strong Beltrami case, studied in Section~\ref{circle}, where the manifold is a circle bundle over a Riemann surface with arbitrary geometry.

If $\curl X$ is not parallel to $X$, then the flow falls into the toroidal case described by Arnold’s theorem. In the compact non-singular setting this leads to a torus-bundle structure, and the additional requirement that $\curl X$ be Killing imposes strong restrictions on the metric. This case is analyzed in Section~\ref{torus}.

\subsection{Circle Bundles over a Surface}\label{circle}
Here we consider the situation where $X$ is a Killing field on a compact orientable three-dimensional manifold $(M,g)$ for which $\curl{X}$ is also a Killing field that is everywhere parallel to $X$. Proposition \ref{beltramilemma} shows that in fact $X$ must be nowhere zero and of constant length. It is possible that such a Killing field could be irregular, such as $X = a\,\partial_x+b\,\partial_y+c\,\partial_z$ on the flat torus $\mathbb{T}^3$, with $\inert X=0$; if the coefficient ratios are irrational, the orbits of $X$ can be dense. The structure theory is simpler if we assume $X$ is regular, with closed orbits all of the same length, so that we have a circle bundle over a compact orientable surface $\Sigma$ with some integer Euler class.

If the Euler class is zero, the bundle is trivial, and $M= \Sigma\times S^1$ with the product Riemannian metric. In this case $X$ is a parallel field (since it is both Killing and harmonic), so that $\lambda=0$ in Theorem \ref{simultaneous eigenvalue to euler solution} and the time dependence is essentially trivial (the action by translation in the $S^1$ variable which does not twist anything, as in Remark \ref{euler trivial solutions}). To get genuinely nontrivial time dependence, we need nonzero Euler class.

We will first describe the general classification of circle bundles over compact manifolds admitting Killing fields that are also curl eigenfields in Section \ref{circle}. Using tools from contact geometry we show that for an arbitrary Riemann surface $\Sigma$ and an arbitrary nonzero integer, there is a circle bundle $M$ with this structure. In Section \ref{section3dexamples}, we begin with Example \ref{spheresection} where we consider the Hopf-fibration of $S^3$ over $S^2$ and give explicit formulas for exact non-steady solutions of \eqref{euler} in this setting. We then extend the Chandrasekhar-Kendall construction of curl eigenfields on any three-dimensional manifold with this structure, in terms of eigenfunctions of the Laplacian on $M$. We conclude by discussing the flat and negatively curved analogues.

\subsubsection{Geometric theory of circle bundles}\label{circlebundles}
Here we will show that given a unit length Killing vector field $X$ which is also a curl eigenfield on a compact $3$-manifold and has periodic orbits all of length $2\pi$, we get a Sasakian contact structure which automatically gives a Boothby-Wang circle fibration over a surface. All such structures are classified by the Euler class of the bundle, and the geometry on the surface is arbitrary. We refer to \cite{blair} and \cite{grabowska2025regularity} for details on the Boothby-Wang fibration and the definitions of Sasakian manifolds, and to \cite{Nicolaescu} and \cite{peralta2021energy} for the relationship to curl eigenfields.

\begin{theorem}\label{boothbywangprop}
    Suppose $M$ is a compact oriented Riemannian $3$-manifold and that $X$ is a unit-length Killing vector field on $M$ with orbits that all have length $2\pi$, such that $\curl{X} = 2\ce X$ for some nonzero $\ce\in\mathbb{R}$. Then there is a Boothby-Wang fibration over a compact oriented surface $\proj\colon M\to\Sigma$ with some nonzero Euler class $\eulerclass\in \mathbb{Z}$, with $\proj$ a Riemannian submersion, and $\pi \eulerclass = \ce \mathrm{Area}(\Sigma)$.

    Conversely given a compact oriented surface $\Sigma$ with any Riemannian metric and any nonzero integer $\eulerclass$, there is a compact oriented $3$-manifold $M$, a map $\proj\colon M\to\Sigma$, and a contact form $\eta$ on $M$ such that $\proj$ is a Riemannian submersion, $\omega$ is the connection $1$-form of the bundle, the Reeb field $X=\eta^{\sharp}$ is unit-length and Killing with orbits of length $2\pi$, and $\curl{X} = 2\ce X$ for $\ce = \pi\eulerclass/\mathrm{Area}(\Sigma)$.
\end{theorem}

\begin{proof}
    First suppose $M$ has a unit-length curl eigenfield $X$ which is also Killing. 
    If we had $\ce=1$, then $M$ would be a Sasakian manifold with contact form $X^{\flat}$ and associated Riemannian metric, with $X$ the Reeb field. Since $\ce$ is nonzero, we can change its sign to positive (if needed) by flipping the orientation of $M$, and change its magnitude by rescaling the metric. If we do this, then the fact that $X$ has orbits all of the same length means that it is a regular Reeb field, and the Boothby-Wang theorem~\cite{blair} shows that $M$ must be a circle fibration over a compact oriented surface $\Sigma$, with $X^{\flat}$ as the connection $1$-form and $dX^{\flat}=\proj^*\Omega$ for a $2$-form satisfying $\int_{\Sigma} \Omega = 2\pi \eulerclass$.  Furthermore the projection $\proj$ is a Riemannian submersion (this is true regardless of the rescaling). In addition the fact that $\curl{X}=2\ce X$ means that $dX^{\flat} = 2\ce\star\! X^{\flat}$, while $\star X^{\flat}$ is the horizontal area form $\proj^*dA_{\Sigma}$; thus $\Omega = 2\ce\, dA_{\Sigma}$, and we see that $\pi \eulerclass = \ce \mathrm{Area}(\Sigma)$. 

  For the converse, we can take any compact, oriented surface $\Sigma$ with a given Riemannian metric and corresponding area $2$-form $dA_{\Sigma}$. Choose a nonzero integer $\eulerclass$, and choose the symplectic form on $\Sigma$ to be $\Omega = 2 \ce dA_{\Sigma}$ for a constant $\ce$ so that $\int_{\Sigma} \Omega = 2\pi \eulerclass$. Then Boothby-Wang also demonstrated that there is a contact manifold $M$ and a circle fibration $\proj\colon M\to \Sigma$ such that the contact form $\eta$ on $M$ is the connection $1$-form of the bundle, with $d\eta = \Omega$. Define the Riemannian metric on $M$ by $$g_M = \eta\otimes \eta + \proj^*g_{\Sigma},$$ which is automatically a Riemannian submersion. Let $X=\eta^{\sharp}$ be the Reeb field of $\eta$. Then $X$ is unit and a Killing field under the new metric and has $2\pi$-periods by construction, and the condition $d\eta = \Omega = 2\ce \star \eta$ translates into $X$ being a curl eigenfield. 
\end{proof}

\begin{remark}
    More explicitly, we can choose geodesic polar coordinates on $\Sigma$ so that the metric is given by $g_{\Sigma} = dr^2 + h(r,\sigma)^2 \, d\sigma^2.$
    Then the area form is $dA_{\Sigma} = h(r,\sigma) \, dr\wedge d\sigma$, and a $1$-form $\eta$ on $M$ with the desired properties is given locally by $$\eta = d\theta - 2\ce \stream_{\sigma} \, dr + 2\ce \stream_r\, d\sigma$$ with $\stream$ solving the Laplace equation $\stream_{rr}+\stream_{\sigma\sigma} = h(r,\sigma)$. The metric then takes the Kaluza-Klein form 
    $$ ds^2 = \big( d\theta  - 2\ce \stream_{\sigma} \, dr + 2\ce \stream_r\, d\sigma\big)^2 + dr^2 + (\stream_{rr}+\stream_{\sigma\sigma})^2\,d\sigma^2,$$
    with freedom to specify the nonzero parameter $\ce$ and the function $f$ arbitrarily in a neighborhood.

    For example, choosing $\varphi$ so that $\varphi_r(r,\sigma) = 2\scurv(\tfrac{r}{2})^2$, makes $\varphi_{rr}(r,\sigma) = \scurv(r)$, in terms of the generalized trigonometric functions $s_\kappa$ from \eqref{generalizedtrig}.
    The metric is then given locally by
    $$ ds^2 = \big(d\theta + 2 \scurv(\tfrac{r}{2})^2\, d\sigma\big)^2 + dr^2 + \scurv(r)^2 \, d\sigma^2.$$
    This is precisely the metric given in \cite{lichtenfelz2022axisymmetric} in Corollary 6.11, describing Hopf-like fibrations of left-invariant metrics on one of the Lie groups $SO_3(\mathbb{R})$, the Heisenberg group, or $SL_2(\mathbb{R})$ over the respective surfaces of constant curvature $k$. There the scaling was chosen so that $\ce=\tfrac{1}{2}$.
\end{remark}

\subsubsection{Curl eigenfields in a circle bundle}\label{section3dexamples}
To obtain non-steady solutions of the Euler equations \eqref{euler} in this context via Theorem \ref{simultaneous eigenvalue to euler solution}, we need to solve the equations
\begin{equation}
    [\curl X, \ef] = i \kev \curl \ef, \qquad \curl \ef = \iev \ef, \qquad [X, \ef] = i \aev \ef,
\end{equation}
which, when $\curl X = 2 \ce X$, and $\aev=n$ an integer, reduces to solving
\begin{equation}\label{rossby3d}
    \curl \ef = \iev \ef, \qquad [X, \ef] = i n \ef,
\end{equation}
and using $\lambda = \frac{2n\ce}{\alpha}$.

We begin with the case of the round three-sphere, where a convenient basis of curl eigenfields is already known.

\begin{example}[Generalized Rossby-Haurwitz Waves on the Round Three-Sphere]\label{spheresection}
The Hopf fibration is the most well-known nontrivial circle bundle, mapping $S^3$ to $S^2$. Given the Hopf fields
\begin{equation*}
    \begin{aligned}
    e_1 &= -y \partial_x + x \partial_y - w \partial_z + z \partial_w, \\
    e_2 &= -z\partial_x + w\partial_y + x\partial_z -y \partial_w, \\
    e_3 &= -w\partial_x -z\partial_y + y\partial_z + x\partial_w.
    \end{aligned}
\end{equation*}
we borrow a basis of curl eigenfields from \cite[Appendix A]{benn2025singular} which is convenient for our purposes. In particular,
\begin{equation*}
    \C \otimes \ff(S^3) = \bigcup_{k=0}^\infty \bigcup_{m=0}^k \mathcal{E}_k^m \ \cup \ \bigcup_{k=2}^\infty \bigcup_{m=1}^{k-1} \mathcal{F}_k^m
\end{equation*}
such that for $k\geq1$ if $\ef \in \mathcal{E}_k^m$ then $\curl \ef = (k+2)v$ and $[\ef,e_1] = -i(2m-k)\ef$. Hence, for $\ef \in \mathcal{E}_k^m$, in the context of Theorem \ref{simultaneous eigenvalue to euler solution}, letting $u_0 = e_1$, we have
\begin{equation*}
    K_{e_1} \ef = -i \frac{2(2m-k)}{k+2} \ef, \qquad \curl \ef = (k+2) \ef, \qquad [e_1, \ef] = i (2m-k) \ef.
\end{equation*}
Note that corresponding solutions of the Euler equations generated by these basis elements are stationary\footnote{The solutions generated by the elements of $\mathcal{E}_0^0$ are all stationary.} if and only if $\freq_{k,m} = (2m-k) \left(\frac{2}{k+2}-1\right) = 0$. That is, if $m = \frac{k}{2}$. Hence, the simplest non-stationary solutions come from setting $k=1$ and $m=0$, which gives $\freq_{1,0} = \frac{1}{3}$. We briefly recall the construction of the basis.

Identifying $\mathbb{R}^4 \simeq \mathbb{C}^2$ we introduce complex coordinates
\begin{equation*}
\alpha = x + iy,\qquad \bar{\alpha} = x - iy,\qquad \beta = z + iw,\qquad \bar{\beta} = z - iw
\end{equation*}
and define the family of homogeneous polynomials in the formal variables $z_1, z_2$
\begin{equation*}
F_{k}^m(\alpha, \bar{\alpha}, \beta, \bar{\beta})(z_1, z_2) = (\alpha z_1 + \beta z_2)^m(- \bar{\beta}z_1 + \bar{\alpha}z_2)^{k-m},
\end{equation*}
for integer $k \geq 0$ and $0 \leq m \leq k$ along with, for $0 \leq j \leq k$, the implicitly defined $Q^m_{kj}(\alpha, \bar{\alpha}, \beta, \bar{\beta})$ given by 
\begin{equation}\label{eq_qkmj_def}
F_k^m(\alpha, \bar{\alpha}, \beta, \bar{\beta})(z_1, z_2) = \sum_{j = 0}^k Q^m_{kj}(\alpha, \bar{\alpha}, \beta, \bar{\beta})z_1^jz_2^{k-j}.
\end{equation}
The $Q^m_{kj}$ are eigenfunctions of the Laplacian on $S^3$ with eigenvalues $k(k+2)$.

The simplest basis element\footnote{See \cite[Appendix A]{benn2025singular} for the definitions of the basis elements.} from $\mathcal{E}_1^0$ is
\begin{equation*}
    v_{112}^0 = Q_{11}^0(e_2 - ie_3) = (x-iy)(e_2-ie_3)
\end{equation*}
and the corresponding real solution is the restriction to $S^3$ of
\begin{equation}\label{explicit 3D RHW}
    U_{\mathfrak{R}}(t) = e_1 + \left( xe_2 - ye_3 \right)\cos\left(\frac{t}{3}\right) + \left(ye_2 + xe_3\right)\sin\left(\frac{t}{3}\right)
\end{equation}
which, as noted before, is of a similar spirit to the two-dimensional Rossby-Haurwitz waves given in \eqref{explicit 2D RHW}.
\end{example}

We now outline a recipe for constructing fields on manifolds of the type described in Theorem \ref{boothbywangprop} satisfying the conditions \eqref{rossby3d}---and thus the assumptions of Theorem \ref{simultaneous eigenvalue to euler solution}---from eigenfunctions of the Laplacian. It is a generalization of the Chandrasekhar-Kendall~\cite{chandrasekhar1957force} construction and may be viewed as a converse to Proposition 1 in \cite{peralta2021energy}, where it is shown that (in our notation) if $\ef$ is a curl eigenfield with eigenvalue $\iev$, then $g(X,\ef)$ is a Laplacian eigenfunction with eigenvalue $\delta^2 = \iev(\iev-2\ce)$.

\begin{lemma}\label{cklemma}
    Let $(M,g)$ be a three-dimensional Riemannian manifold which admits a Killing field $X$ such that $\curl{X} = 2\ce X$ for some $\ce\in\mathbb{R}$.
    If $f$ is a complex-valued function on $M$ with constants $\delta>0$ and $\aev \in \R$ such that
    \begin{equation}\label{curleigenfieldfconditions}
        \Delta f = \delta^2 f, \qquad X(f) = i\aev f,
    \end{equation}
    then, letting $\iev_{\pm} := \ce \pm \sqrt{\ce^2+\delta^2}$, the vector fields 
    \begin{equation*}
        \ef_{\pm} = \iev_{\pm}^2 f X + \iev_{\pm} \nabla f\times X + i\aev \nabla f 
    \end{equation*}
    are curl eigenfields on $M$ with eigenvalues $\iev_{\pm}$.
    
    Furthermore, any such $\ef_{\pm}$ yields an exact solution of the Euler equations \eqref{euler} via Theorem \ref{simultaneous eigenvalue to euler solution}.
\end{lemma}

\begin{proof} For any $f$ satisfying the assumptions of \eqref{curleigenfieldfconditions}, consider the vector fields
    \begin{equation}\label{curleigenbasis}
    E_1 = fX, \qquad E_2 = \nabla f\times X, \qquad E_3 = \nabla f.
    \end{equation}
    Taking the $\curl$  of these fields yields
    \begin{align*}
        \curl(E_1) &= \nabla f\times X + f \curl{X} = 2\ce E_1 + E_2 \\
        \curl(E_2) &= (\diver{X}) \nabla f + (\Delta f) X + [X,\nabla f] = \delta^2 f X + i\aev \nabla f = \delta^2 E_1 + i\aev E_3 \\
        \curl(E_3) &= 0.
    \end{align*}
where we have used the fact that $\diver{X}=0$ and that, since $X$ is a Killing field, its action on functions commutes with the gradient, giving $[X,\nabla f] = \nabla X(f)$. Hence curl on the span of these vector fields is given by the matrix
\begin{equation*}
    \curl\vert_{\mathrm{span}\{E_i\}}
    =
    \left(
    \begin{matrix} 
    2\ce & \delta^2 & 0 \\
    1 & 0 & 0 \\
    0 & i\aev & 0
    \end{matrix}
    \right),    
\end{equation*}
whose eigenvalues are $\alpha_{\pm}$ and $0$ with corresponding eigenvectors $\ef_{\pm}$ and $E_3$ respectively.
\end{proof}

\begin{remark}
    Note that the existence of $f$ satisfying \eqref{curleigenfieldfconditions} comes from the fact that $X$ is a Killing field and hence commutes with the Laplacian. Hence, we can find an $L^2$ basis of simultaneous eigenfunctions. Each one generates an exact solution of the three-dimensional Euler equations via Theorem \ref{simultaneous eigenvalue to euler solution}.
\end{remark}

Armed with this, we now consider the flat and negatively curved analogues of Example \ref{spheresection}.

\begin{example}[The Heisenberg Manifold]
    The Heisenberg group is the space of upper-triangular $3\times 3$ matrices with $1$ on the diagonal, which is homeomorphic to $\mathbb{R}^3$. The set of such matrices with integer entries is a discrete subgroup, and the quotient is a compact $3$-manifold called the Heisenberg manifold. It is a circle bundle over the flat torus $\T^2$. Eigenfunctions of the Laplacian can be worked out as in Tolimieri~\cite{tolimieri1977analysis} using Hermite functions, and from these one could construct curl eigenfields in order to obtain explicit solutions. The formulas involve infinite sums, however, and are not nearly as elementary as in the case of the $3$-sphere, so we will not discuss them further here, although the interested reader could calculate them.
\end{example}

\begin{example}[Hyperbolic Circle Bundles]
    Let $\Sigma$ be a compact orientable surface of genus $g\ge 2$. By uniformization, $\Sigma$ admits a metric of constant curvature $-1$, realized as a quotient of the hyperbolic plane by a discrete group of M\"obius transformations. In their study of the fast dynamo problem, Arnold and Khesin \cite[Chapter V.4.D]{arnold2021topological} illustrated that the unit tangent bundle of such a surface is a nontrivial circle bundle of the type appearing in Theorem~\ref{boothbywangprop}, and carries natural curl eigenfields arising from this geometry. The same geometry gives examples that fit here, although the eigenfunctions of the Laplacian are harder to compute explicitly. The Euler class $k$ from Theorem \ref{boothbywangprop} for a unit tangent bundle happens to coincide with the Euler characteristic, so $k=\chi(\Sigma)=2-2g$.  The area of $\Sigma$ is given by the Gauss-Bonnet theorem as $\mathrm{Area}(\Sigma) = -2\pi\chi(\Sigma) = 4\pi(g-1)$. Thus the curl eigenvalue is $2\epsilon = -1$.
    We refer the interested reader to \cite{arnold2021topological} for details of the geometric construction. Note lastly that this construction also works over $S^2$, giving the bundle with Euler class $k=2$, but it does not work on $\mathbb{T}^2$ since the unit circle bundle is $\mathbb{T}^3$, which is a trivial bundle.
\end{example}

\subsection{Torus Bundles over an Interval}\label{torus}
We now consider the second geometric situation from Section~\ref{3D}, in which the Killing field $X$ and its vorticity $\omega=\curl X$ are not everywhere parallel; and hence are linearly independent on an open set. In this case Arnold’s structure theorem suggests a toroidal organization of the flow, and the geometry is naturally expressed in coordinates adapted to the commuting fields $X$ and $\omega$. Our first goal is to determine the local form of the metric under the additional assumption that both fields are Killing. We then use this structure to construct explicit curl eigenfields and corresponding non-steady Euler solutions.

\subsubsection{The form of the metric}
\begin{theorem}\label{transversecurltheorem}
    Let $(M,g)$ be an oriented Riemannian $3$-manifold. If there is a Killing field $X$ such that $\omega=\curl{X}$ is also Killing, and $\{X,\omega\}$ are linearly independent in some neighborhood, then there exist local oriented coordinates $(r,\theta, z)$, a constant $c$ and a positive function $\fnone$ such that the metric is given by
    \begin{equation}\label{3disometriesmetric}
        ds^2 = dr^2 + \Big(\fnone(r) d\theta + \frac{c}{\fnone(r)}dz\Big)^2 + \fnone'(r)^2dz^2
    \end{equation}
    with $X=\partial_\theta$ and $\omega = 2\partial_z$.
\end{theorem}

\begin{proof}
In any neighborhood where $X$ and $\omega$ are linearly independent, they commute by \eqref{Killing Lie bracket}, so we may choose coordinates $(\theta,z)$ such that $X = \partial_{\theta}$ and $\omega = 2\,\partial_{z}$, where the normalization of $\omega$ is chosen for convenience. By formula \eqref{killingcross}, the function $\lvert X\rvert_g^2$ is not constant in this neighborhood. Let $\Sigma$ be a regular level surface of $\abs{X}_g$. Since $X$ and $\omega$ are linearly independent, the vector field $X \times \omega$ is nowhere vanishing and orthogonal to both. We define a coordinate $r$ by taking unit-speed geodesics orthogonal to $\Sigma$ in the direction of $X \times \omega$. By construction,
\begin{equation*}
    g(\partial_r,\partial_r)=1, \qquad \nabla_{\partial_r}\partial_r = 0,
\end{equation*}
cf. Petersen~\cite[Section 5.6]{petersen2006riemannian}. We claim that $g(\partial_r,X)=g(\partial_r,\omega)=0$. Indeed,
\begin{equation*}
\partial_r\big(g(\partial_r,X)\big)
= g(\nabla_{\partial_r}\partial_r,X) + g(\partial_r,\nabla_{\partial_r}X).
\end{equation*}
The first term vanishes since $\nabla_{\partial_r}\partial_r=0$, and the second vanishes because $X$ is Killing. Since $g(\partial_r,X)=0$ along $\Sigma$, it follows that $g(\partial_r,X)\equiv 0$. The same argument applies to $\omega$. Hence
\begin{equation*}
g(\partial_r,\partial_\theta)=g(\partial_r,\partial_z)=0.
\end{equation*}

By Proposition~\ref{beltramilemma}, the quantity $g(\partial_\theta,\partial_z)$ is constant; denote it by $c$. Since the metric is invariant under both Killing fields, its coefficients depend only on $r$. It follows that the metric can be written in the form
\begin{equation*}
    ds^2 = dr^2 + \left(\fnone(r)d\theta + \frac{c}{\fnone(r)}dz\right)^2 + \fntwo(r)^2dz^2,
\end{equation*}
for positive functions $\fnone(r)$ and $\fntwo(r)$.
An orthonormal basis of $1$-forms is given by
\begin{equation*}
   \xi^1 = dr, \quad \xi^2 = \fnone(r)\,d\theta + \frac{c}{\fnone(r)}\, dz, \quad \xi^3 = \fntwo(r)\,dz. 
\end{equation*}

We compute
\begin{equation*}
X^\flat = \partial_\theta^\flat = \fnone(r)^2\,d\theta + c\,dz = \fnone(r)\,\xi^2,
\end{equation*}
and hence
\begin{equation*}
    dX^\flat = 2\fnone(r)\fnone'(r)\,dr \wedge d\theta.
\end{equation*}
Expressing this in the orthonormal coframe gives
\begin{equation*}
    dX^\flat =
2\fnone'(r)\left(
\xi^1 \wedge \xi^2
+
\frac{c}{\fnone(r)\fntwo(r)}\,\xi^3 \wedge \xi^1
\right),
\end{equation*}
so that
\begin{equation*}
    \star dX^\flat =
2\fnone'(r)\left(
\xi^3
+
\frac{c}{\fnone(r)\fntwo(r)}\,\xi^2
\right).
\end{equation*}
On the other hand,
\begin{equation*}
    \omega^\flat = 2\,\partial_z^\flat =
2\left(
\fntwo(r)\,\xi^3 + \frac{c}{\fnone(r)}\,\xi^2
\right).
\end{equation*}
Comparing $\star dX^\flat = \omega^\flat$ yields $\fntwo(r) = \fnone'(r)$, and substituting this into the metric gives \eqref{3disometriesmetric}.

Conversely, for any constant $c$ and positive function $\fnone$, the vector field $\partial_\theta$ is Killing, as the metric components are independent of $\theta$. We can then compute that $\curl(\partial_\theta)=2\,\partial_z$ which, since the metric coefficients are also independent of $z$, is a linearly independent Killing field. This completes the proof.
\end{proof}

\begin{remark}\label{czeroremark}
Although we have $\curl{\partial_{\theta}} = 2\,\partial_z$, the curl of $\partial_z$ is not so simple:
\begin{equation}\label{curldzcmetric}
\curl{\partial_z} = 2\left( \frac{c^2}{\fnone(r)^4} - \frac{\fnone''(r)}{\fnone(r)}\right) \, \partial_{\theta}.
\end{equation}
In the special case where $c=0$ and $\fnone(r) = \scurv(r)$ as in \eqref{generalizedtrig}, this metric becomes 
$$ ds^2 = dr^2 + \scurv(r)^2 \, d\theta^2 + \scurv'(r)^2 \, dz^2,$$
which is the standard constant curvature metric in polar coordinates for $\curv\in \{1,0,-1\}$ on $S^3$, $\mathbb{R}^3$, or hyperbolic $3$-space. Here the Killing field is the usual axisymmetric rotation $\partial_{\theta}$, which vanishes along the curve $r=0$. In this case we have $\curl{\partial_{\theta}} = 2\partial_z$ while $\curl{\partial_z} = 2\curv \partial_{\theta}$. When $\curv=1$ on the $3$-sphere, we see that $\partial_{\theta}+\partial_z$ is exactly the Hopf field from Example \ref{spheresection}. Hence the formulas here may be used to obtain new solutions analogous to the ones presented there.
\end{remark}
\begin{remark}\label{cnonzeroremark}
The metric \eqref{3disometriesmetric} has volume form $dV = \fnone(r) \fnone'(r) \,dr\wedge d\theta \wedge dz$  regardless of $c$. If we define $\eta = \fnone(r)^2 \, d\theta + c\,dz$, then we can compute 
$$ \eta\wedge d\eta = 2c \fnone(r) \fnone'(r) \, dr \wedge d\theta\wedge dz = 2c \, dV,$$
so that for $c\ne 0$ we have a contact form. The corresponding Reeb field is $\xi = c^{-1} \, \partial_z$, which has squared length $\frac{\fnone'(r)^2}{c^2}$. In the special case $\fnone(r)=r$ and $c=1$, the Reeb field is Killing and unit length, so we have a Sasakian geometry. 
The scalar curvature of the metric \eqref{3disometriesmetric} in general is given by 
$$ R = -\frac{2\fnone'''(r)}{\fnone'(r)} - \frac{4\fnone''(r)}{\fnone(r)} - \frac{2c^2}{\fnone(r)^4}, $$
from which we see that the parameter $c$ genuinely changes the geometry, even in the special case $\fnone(r)=r$. For $\fnone(r) = \sin{r}$ we get $R = 6 - 2c^2 \csc^4{r}$, showing that in general we cannot expect to make this metric well-defined at a point where $\fnone(r)=0$; hence it will typically only make sense on a non-compact manifold or a manifold with boundary. We do not know if there is any closed $3$-manifold that admits such a geometry with $c\ne 0$.
\end{remark}
\subsubsection{Curl eigenfields in a torus bundle}
We now study the geometry of the $c$-metric \eqref{3disometriesmetric}, and construct solutions to the Euler equations \eqref{euler} in this geometry via Theorem \ref{simultaneous eigenvalue to euler solution}. For concreteness suppose that $\theta$ and $z$ are both $2\pi$-periodic, so that $X=\partial_{\theta}$ and $\omega_0 = 2\partial_z$ span a torus for each fixed $r$ with closed trajectories, and assume $M$ is a solid torus bounded by $\radiusone \le r\le \radiustwo$ with both $\fnone(r)$ and $\fnone'(r)$ positive everywhere (which avoids potential singularities where either vanishes).

The following construction is similar to that of the curl eigenfields in Appendix C of \cite{peralta2025asymmetry}, on the flat solid torus (with $c=0$ and $\fnone(r)=r$). Note that we cannot use the technique of Lemma \ref{cklemma} here, since although $\curl{(\partial_{\theta})} = 2\,\partial_z$ by our definitions, the curl of $\partial_z$ does not simplify, cf. \eqref{curldzcmetric}. In particular we would in general not get a finite-dimensional subspace by repeatedly taking curls. Therefore the computations need to be done directly as in the proof of Theorem \ref{transversecurltheorem}. The details of the proof are presented in Appendix \ref{geometric lemmas}.

\begin{theorem}\label{cmetriccurls}
Consider the three-dimensional manifold $[\radiusone,\radiustwo]\times \mathbb{T}^2$ with metric \eqref{3disometriesmetric}. For $m,n \in \Z$, let $\fnthree$ and $\fnfour$ solve the Sturm-Liouville type system
\begin{align}
\alpha \fnone(r) \fnone'(r) \fnthree'(r) &= n\left( m - \frac{cn}{\fnone(r)^2}\right) \, \fnthree(r) + \big( \alpha^2 \fnone(r)^2 - n^2\big) \fnfour(r) \label{geqcmetric} \\
\alpha \fnone(r)\fnone'(r) \fnfour'(r) &= \left( 2c\alpha \, \frac{\fnone'(r)^2}{\fnone(r)^2} + \left( m-\frac{cn}{\fnone(r)^2}\right)^2 - \alpha^2 \fnone'(r)^2 \right) \, \fnthree(r) - n\left( m - \frac{cn}{\fnone(r)^2}\right) \, \fnfour(r),\label{heqcmetric}
\end{align}
with boundary conditions 
\begin{equation}\label{cmetricbcs}
n\fnfour(\radiusone) - \left( m- \frac{cn}{\fnone(\radiusone)^2}\right) \fnthree(\radiusone) = n\fnfour(\radiustwo) - \left( m- \frac{cn}{\fnone(\radiustwo)^2}\right) \fnthree(\radiustwo) = 0,
\end{equation}
the constant $\iev$ implicitly defined and $\fnfive$ given by
\begin{equation*}
    \fnfive(r) = \,\frac{n\fnfour(r)-\big(m-\frac{cn}{\fnone(r)^2}\big)\fnthree(r)}{\alpha \fnone(r)\fnone'(r)}.
\end{equation*}
Then the complex fields
\begin{equation}\label{cmetriccurlfield}
    \ef(r,\theta,z) = e^{in\theta} e^{imz} \bigg( i \fnfive(r) \, \partial_r + 
    \frac{\fnthree(r)}{\fnone(r)^2} \, \partial_{\theta} + \frac{\fnfour(r)}{\fnone'(r)^2} \Big( -\frac{c}{\fnone(r)^2} \, \partial_{\theta} + \partial_z\Big)
    \bigg)
\end{equation}
are curl-eigenfields satisfying the conditions of Theorem \ref{simultaneous eigenvalue to euler solution} with
\begin{equation*}
        K_{\partial_\theta} \ef = -i \frac{2m}{\alpha} \ef, \qquad \curl \ef = \iev \ef, \qquad [\partial_\theta, \ef] = i n \ef.
    \end{equation*}
\end{theorem}

In the simplest case, the curl eigenfields reduce to Chandrasekhar-Kendall modes on the cylinder.

\begin{example}\label{cequalszero}
    On the solid torus $D^2\times S^1$, with the standard flat metric $ds^2 = dr^2 + r^2 \,d\theta^2 + dz^2$ corresponding to $c=0$ and $\varphi(r)=r$ in \eqref{3disometriesmetric}, with $\radiusone=0$ and $\radiustwo=1$, 
    we can see that $\fnfour$ satisfies the Bessel equation 
    $$ \fnfour''(r) + \frac{1}{r} \, \fnfour'(r) + \Big( \alpha^2 - m^2 - \frac{n^2}{r^2}\Big) \fnfour(r) = 0.$$
    The solution is \begin{equation}
        \fnfour(r) = -\beta^2 J_n(\beta r), \qquad \fnthree(r) = 
        \beta \alpha r J_n'(\beta r) + nm J_n(\beta r), \qquad \alpha^2=\beta^2+m^2,
    \end{equation} with $\beta$ chosen so that 
    $$ m\beta J_n'(\beta) + n\alpha \, J_n(\beta)=0. $$
    The corresponding curl eigenfield \eqref{cmetriccurlfield} is precisely the Chandrasekhar-Kendall~\cite{chandrasekhar1957force} field on the cylinder
$$     v(r,\theta,z) = e^{in\theta} e^{imz} \bigg( i\,\frac{n\fnfour(r)-m\fnthree(r)}{\alpha r} \, \partial_r + 
    \frac{\fnthree(r)}{r^2} \, \partial_{\theta} + \fnfour(r) \,\partial_z
    \bigg).$$ 
    Hence the corresponding solution of the 3D Euler equation from Theorem \ref{simultaneous eigenvalue to euler solution}, with $\lambda=\frac{2m}{\alpha}$, is given by 
    \begin{multline*}
        U(t) = \partial_{\theta} + \rho \sin{( mz + n\theta +
        \freq t)} \big( m\beta J_n'(\beta r) + \tfrac{n\alpha}{r} \, J_n(\beta r)\big) \, \partial_r \\
    + \rho \cos{( mz + n\theta +
    \freq t)}
    \Big( \tfrac{1}{r^2} \big( \alpha \beta r J_n'(\beta r) + nm J_n(\beta r)\big) \, \partial_{\theta} - \beta^2 J_n(\beta r) \, \partial_z\Big),
    \end{multline*}
    for any constant $\rho\ge 0$, with $\freq = \frac{2m}{\alpha}-n$.
Kelvin originally derived these solutions in 1880 as solutions to the linearized Euler equation (i.e., valid only approximately for small $\rho$), and Dritschel~\cite{dritschel1991generalized} observed that they also solve the nonlinear Euler equation for any $\rho$.
\end{example}

We now show what happens in the twisted case where $c\ne 0$, with the same simple metric $\fnone(r)=r$.

\begin{example}\label{cnonzero}
Choose $\fnone(r)=r$ in \eqref{3disometriesmetric} with $c\ne 0$ and $0<\radiusone<\radiustwo$. For $n=0$ and $m\in\mathbb{N}$, the system \eqref{geqcmetric}--\eqref{cmetricbcs} becomes
\begin{align*}
    \fnthree'(r) &= \alpha r \fnfour(r) \\
    \fnfour'(r) - \frac{2c \fnthree(r)}{r^3} + mf(r) &= -\frac{\alpha \fnthree(r)}{r} 
\end{align*}
 with boundary conditions $g(\radiusone)=g(\radiustwo)=0$ and 
 $\alpha r \fnfive(r) = -m\fnthree(r)$. Eliminating $\fnfour$, we get 
$$ \fnthree''(r) = \frac{\fnthree'(r)}{r} + \left( \frac{2c\alpha}{r^2} + m^2 - \alpha^2\right) \fnthree(r), \qquad g(\radiusone)=g(\radiustwo)=0.$$
Solutions are given via Bessel functions: 
$$ \fnthree(r) = c_1 r J_{\nu}(k r) + c_2 r Y_{\nu}(kr), \qquad k^2 := \alpha^2 - m^2, \qquad \nu^2 := 2\alpha c+1, $$
and there are nontrivial solutions vanishing at $r=\radiusone$ and $r=\radiustwo$ if and only if 
$$ J_{\nu}(k\radiusone) Y_{\nu}(k\radiustwo) - J_{\nu}(k\radiustwo) Y_{\nu}(k\radiusone) = 0.$$
This becomes an equation to solve for the curl eigenvalue $\alpha$, given $m$ and $c$. 

For example if $c=-\tfrac{3}{10}$ and $m=1$ with $\radiusone=\tfrac{2\pi}{3}$ and $\radiustwo=2\pi$, then $\alpha = \tfrac{5}{4}$ works, since then $k=\tfrac{3}{4}$ and $\nu=\tfrac{1}{2}$, so that the Bessel functions become elementary. The corresponding solutions are multiples of 
$$ f(r) = -4r^{-1/2} \cos{(\tfrac{3r}{4})}, \qquad \fnthree(r) = 5r^{1/2} \cos{(\tfrac{3r}{4})}, \qquad  \fnfour(r) = 
-3r^{-1/2}\sin{(\tfrac{3r}{4})} + 2r^{-3/2} \cos{(\tfrac{3r}{4})}.
$$
Combining Theorem \ref{cmetriccurls} and Theorem \ref{simultaneous eigenvalue to euler solution} to get $\lambda = \frac{2m}{\alpha} = \frac{8}{5}$, we conclude that for any $\rho\ge 0$, 
\begin{multline*} u(t) = 
\partial_{\theta} + 4\rho r^{-1/2}\sin{\big( z + \tfrac{8}{5}t \big)} \cos{\big( \tfrac{3}{4} r\big)} \, \partial_r 
+ 5\rho r^{-3/2} \cos{\big( z + \tfrac{8}{5}t\big)} \cos{\big( \tfrac{3}{4} r\big)} \, \partial_{\theta} \\
+ \rho\cos{\big( z+\tfrac{8}{5} t\big)} \Big( -3r^{-1/2} \sin{\big(\tfrac{3}{4} r\big)} + 2r^{-3/2} \cos{\big( \tfrac{3}{4}r\big)} \Big) \Big( \partial_z + \tfrac{3}{10r^2} \, \partial_{\theta}\Big)
\end{multline*}
is  a non-steady solution of the 3D Euler equation on the solid torus $[\tfrac{2\pi}{3}, 2\pi]\times \mathbb{T}^2$ in the geometry 
$$ ds^2 = dr^2 + r^2 \big( d\theta - \tfrac{3}{10r^2} \, dz\big)^2 + dz^2.$$
This is probably the simplest explicit solution one could write down in such a geometry; for other values with $n=0$ we would get solutions in terms of Bessel functions, while for $n\ne 0$ (corresponding to solutions that are not axially symmetric), the ODE is much more complicated, though still in principle solvable.
\end{example}

\section{The Euler-Arnold Framework and Gerneralized Coriolis Force}\label{general framework}
The purpose of this section is to motivate Theorem~\ref{simultaneous eigenvalue to euler solution} and exhibit it as an application of a more general correspondence between Euler-Arnold dynamics and a class of modified equations incorporating a generalized Coriolis force.

As recalled in the introduction, the geometric approach to hydrodynamics initiated by Arnold~\cite{arnold2013differential} models the configuration space of an ideal fluid as the Lie group of volume-preserving diffeomorphisms. This group is endowed with a right-invariant Riemannian metric obtained by equipping its Lie algebra—the space of divergence-free vector fields—with the $L^2$ inner product corresponding to kinetic energy. When reduced to the Lie algebra, the associated geodesic equation coincides with the Euler equations of ideal hydrodynamics~\eqref{euler}. While this framework motivates our results, we now work in a more general setting.

Let $\G$ be a Lie group with Lie algebra $\g$ and identity element $e$. For $\sigma\in\G$, define the right and left translations
\begin{equation*}
R_\sigma(\xi)=\xi\circ\sigma, \qquad L_\sigma(\xi)=\sigma\circ\xi.
\end{equation*}
The adjoint action of $\G$ on $\g$ is defined by
\begin{equation*}
\Ad_\sigma v = d_\sigma R_{\sigma^{-1}}\, d_e L_\sigma\, v,
\end{equation*}
and the adjoint action of $\g$ on itself is given, for $u\in\g$, by
\begin{equation*}
\ad_u v = \left.\frac{d}{dt}\right|_{t=0} \Ad_{\sigma(t)} v,
\end{equation*}
where $\sigma(t)$ is any smooth curve in $\G$ with $\sigma(0)=e$ and $\dot\sigma(0)=u$.

If $\g$ is equipped with an inner product $\langle\cdot,\cdot\rangle$ we may extend this via right-translation to a Riemannian metric on $\G$ by
\begin{equation}\label{general right-invariant metric}
\ip{U,V}_\sigma = \ip{d_e R_\sigma^{-1} U, ~d_e R_\sigma^{-1} V}, \qquad U,V\in T_\sigma\G.
\end{equation}
We assume that this metric admits a well-defined Riemannian exponential map
\begin{equation*}
\exp_e : \mathcal U\subset\g\to\G
\end{equation*}
that is a diffeomorphism from a neighborhood of $0\in\g$ onto a neighborhood of $e\in\G$.

The corresponding coadjoint operators are defined by
\begin{equation*}
\ip{\Adstar_\sigma u, v} = \ip{u, \Ad_\sigma v}, \qquad \ip{\adstar_u v, w}= \ip{v, \ad_u w}, \quad \sigma\in\G, ~u,v,w\in\g.
\end{equation*}

\begin{remark}
    Note that for fixed $v\in \g$ the operator $u\mapsto\adstar_uv$ is skew-adjoint with respect to $\ip{\cdot,\cdot}$ and therefore has purely imaginary spectrum on the complexified Lie algebra $\C \otimes \g$. This observation will play a key role in the later construction of solutions.
\end{remark}

With these definitions, the geodesic equation on $\G$, when reduced to the Lie algebra, takes the form
\begin{align}\label{euler-arnold}
\begin{split}
    \partial_t u &= -\adstar_u u,\\
    u &= (d_e R_\gamma)^{-1}\dot\gamma,\\
    u(0) &= u_0,
\end{split}
\end{align}
where $\gamma(t)=\exp_e(tu_0)$ is the unique geodesic satisfying $\gamma(0)=e$ and $\dot\gamma(0)=u_0$. Equations of the form~\eqref{euler-arnold} are known as Euler--Arnold equations and imply the conservation law
\begin{equation}\label{general coadjoint conservation law}
    \Adstar_{\gamma(t)} u(t) = u_0.
\end{equation}

We now introduce the main correspondence underlying our construction of solutions. While the formulation may at first appear somewhat ad hoc, it will later be motivated by a natural perturbative argument.

\enlargethispage{-2\baselineskip}

\begin{theorem}\label{Euler and Euler-Coriolis correspondence}
    Let $\G$ be a Lie group with Lie algebra $\g$, equipped with a right-invariant metric $\ip{\cdot,\cdot}$ admitting a smooth Riemannian exponential map $\exp_e \colon \mathcal U \subset \g \to \G$ that is a diffeomorphism from a neighborhood of $0\in\g$ onto a neighborhood of $e\in\G$.
    Let $X \in \g$ be such that
    \begin{equation}\label{generalized Killing}
        \adstar_X = -\ad_X
    \end{equation}
    and denote its flow by $\gamma_X$. Then a curve in $\g$ given by
    \begin{equation}\label{equivalence}
        U(t) = X + \Ad_{\gamma_X(t)}V(t)
    \end{equation}
    solves the Euler-Arnold equations \eqref{euler-arnold} if and only if the field $V(t)$ solves
    \begin{equation}\label{euler-arnold-coriolis}
        \partial_t V = -\adstar_VV - \adstar_VX.
    \end{equation}
    Moreover, $U$ is stationary if and only if
    \begin{equation}\label{stationary condition}
        \partial_t V = \adstar_X V.
    \end{equation}
\end{theorem}

\begin{remark}\label{equivalence remark}
The correspondence in Theorem~\ref{Euler and Euler-Coriolis correspondence} shows that the Euler-Arnold equation \eqref{euler-arnold} and the modified equation \eqref{euler-arnold-coriolis} are equivalent up to a time-dependent change of variables. In particular, solutions of one system are in one-to-one correspondence with solutions of the other via \eqref{equivalence}. Thus, from a dynamical point of view, the two equations encode the same information: qualitative and quantitative properties of solutions (such as existence, regularity, or conserved quantities) may be transferred directly between them. In this sense, the additional term $\adstar_V X$ can be interpreted not as introducing genuinely new dynamics, but rather as expressing the Euler-Arnold equation in a rotating frame generated by $X$.
\end{remark}

\begin{remark}
    For ideal hydrodynamics, the condition $\adstar_X=-\ad_X$ is equivalent to requiring that $X$ be a Killing field on the fluid domain. In this case, one may verify (at least locally in time) that the flow generated by $U$ is obtained by composing the flow of $V$ with the isometric flow generated by $X$, i.e., $\gamma_U = \gamma_X \circ \gamma_V$. 
    On the round two-sphere, taking $X=\partial_\theta$, the term $\adstar_V X$ reduces to the classical Coriolis force $-2\Delta^{-1}\partial_\theta V$. This motivates our terminology of a \emph{generalized Coriolis force}. Related ideas, including a notion of magnetic force, appear for geodesic motion on central extensions of diffeomorphism groups, cf. \cite{arnold2021topological, marsden2007hamiltonian, vizman2008geodesicequations}. However here we work directly on the group and interpret the additional terms as arising from a change to a rotating frame generated by a Killing field.
\end{remark}

Before turning to the proof of Theorem~\ref{Euler and Euler-Coriolis correspondence}, we record a technical lemma concerning generalized Killing fields.

\begin{lemma}\label{general Killing lemma}
    Let $X \in g$ satisfy $\adstar_X = -\ad_X$ and denote its flow by $\gamma_X$. Then
    \begin{enumerate}
        \item[(i)] $X$ is a steady-state solution of the Euler-Arnold equation \eqref{euler-arnold}. \\
        \item[(ii)] $\Adstar_{\gamma_X(t)}X = \Ad_{\gamma_X(t)} X = X$. \\
        \item[(iii)] $\adstar_{\Ad_{\gamma_X(t)}v}\Ad_{\gamma_X(t)}w = \Ad_{\gamma_X(t)}\adstar_v w$ for all $v,w \in \g$. \\
        \item[(iv)] $\adstar_{X}\Ad_{\gamma_X(t)}v = \Ad_{\gamma_X(t)}\adstar_X v$ for all $v \in \g$. \\
        \item[(v)] $\adstar_{\Ad_{\gamma_X(t)}v} X = \Ad_{\gamma_X(t)}\adstar_v X$ for all $v \in \g$.
    \end{enumerate}
\end{lemma}

\begin{proof}
    The first assertion follows immediately from $\ad^*_X X = -\ad_X X = 0$. We recall that the adjoint action $\Ad$ satisfies the homomorphism identity
    \begin{equation}
        \Ad_\sigma\,\ad_u v = \ad_{\Ad_\sigma u}\,\Ad_\sigma v, \qquad \sigma\in\G,~ u,v\in\g
    \end{equation}
    and, if $\sigma(t)$ is a smooth curve in $\G$ with associated velocity $\partial_t \sigma(t) = d_e R_{\sigma(t)} u(t)$, then the derivative of the adjoint action is given by
    \begin{equation}
        \frac{d}{dt}\Ad_{\sigma(t)} v = \ad_{u(t)}\,\Ad_{\sigma(t)} v, \qquad v\in\g.
    \end{equation}

    Taking adjoints with respect to the right-invariant metric yields the corresponding identities for the coadjoint action
    \begin{equation}\label{Adstar homomorphism}
    \Adstar_\sigma\,\adstar_{\Ad_\sigma u} v = \adstar_u\,\Adstar_\sigma v, \qquad \sigma\in\G,~ u,v\in\g,
    \end{equation}
    and
    \begin{equation}\label{Adstar derivative}
        \frac{d}{dt}\Adstar_{\sigma(t)} v = \Adstar_{\sigma(t)}\,\adstar_{u(t)} v, \qquad v\in\g.
    \end{equation}
    Next, for each $v \in \g$, we define the operator
    \begin{equation}
        \Lambda_v(t) = \Adstar_{\exp_e(tv)} \Ad_{\exp_e(tv)}.
    \end{equation}
    In particular, $\Lambda_X(t) = \Adstar_{\gamma_X(t)}\Ad_{\gamma_X(t)}$ and we have
    \begin{align*}
        \frac{d}{dt} \Lambda_X(t) v &= \frac{d}{dt} \left( \Adstar_{\gamma_X(t)}\Ad_{\gamma_X(t)} v \right) \\
        &= \Adstar_{\gamma_X(t)}\adstar_X \Ad_{\gamma_X(t)} v + \Adstar_{\gamma_X(t)}\ad_X \Ad_{\gamma_X(t)} v \\
        &= - \Adstar_{\gamma_X(t)}\ad_X \Ad_{\gamma_X(t)} v + \Adstar_{\gamma_X(t)}\ad_X \Ad_{\gamma_X(t)} v \\
        &= 0.
    \end{align*}
    Hence, as $\Lambda_X(0) = \mathrm{Id}$, we have $\Lambda_X(t) = \mathrm{Id}$ for all $t$. Combining this with the coadjoint conservation law \eqref{general coadjoint conservation law}
    \begin{equation*}
        \Adstar_{\gamma_X(t)} X = X
    \end{equation*}
    yields the second assertion of the lemma.

    The remaining assertions follow directly from \eqref{Adstar homomorphism}, the second assertion and the identity $\Lambda_X(t) = \Adstar_{\gamma_X(t)}\Ad_{\gamma_X(t)} = \mathrm{Id}$.
\end{proof}

With these preparations in place, we now prove Theorem~\ref{Euler and Euler-Coriolis correspondence}.

\begin{proof}[Proof of Theorem~\ref{Euler and Euler-Coriolis correspondence}]
    Suppressing $t$, we begin by taking a derivative in time of
    \begin{equation*}
        U(t) = X + \Ad_{\gamma_X}V
    \end{equation*}
    which gives
    \begin{equation*}
        \partial_t U = \Ad_{\gamma_X}\partial_tV + \ad_X \Ad_{\gamma_X}V.
    \end{equation*}
    From Lemma \ref{general Killing lemma} we have
    \begin{equation*}
        \ad_X \Ad_{\gamma_X(t)}V(t) = - \adstar_X \Ad_{\gamma_X(t)}V(t) = -\Ad_{\gamma_X(t)}\adstar_XV(t)
    \end{equation*}
    and hence we have
    \begin{equation*}
        \partial_t U = \Ad_{\gamma_X} \big( \partial_tV - \adstar_X V \big).
    \end{equation*}
    Note that this immediately gives the stationary condition \eqref{stationary condition}. Next, we compute the right-hand side of the Euler-Arnold equation \eqref{euler-arnold}
    \begin{align*}
        -\adstar_UU &= -\adstar_{X + \Ad_{\gamma_X}V}\big(X + \Ad_{\gamma_X}V\big) \\
        &= -\adstar_XX - \adstar_X\Ad_{\gamma_X}V - \adstar_{\Ad_{\gamma_X}V}X - \adstar_{\Ad_{\gamma_X}V}\Ad_{\gamma_X}V.
    \end{align*}
    Noting again that $-\adstar_XX = \ad_XX = 0$ and using Lemma \ref{general Killing lemma} we have
    \begin{equation*}
        -\adstar_UU = -\Ad_{\gamma_X} \big( -\adstar_XV -\adstar_VX - \adstar_VV \big).
    \end{equation*}
    Combing this with the above, we see that
    \begin{equation*}
        \partial_tU = -\adstar_UU
    \end{equation*}
    if and only if
    \begin{equation*}
        \partial_tV = - \adstar_VV -\adstar_VX
    \end{equation*}
    as claimed.
\end{proof}

We now outline a mechanism for constructing solutions of~\eqref{euler-arnold-coriolis}, and hence solutions to the Euler–Arnold equation~\eqref{euler-arnold}. The spectrum of the operator $\ef \mapsto \adstar_\ef X$ naturally provides candidates for such solutions. Since it is skew-adjoint, its eigenvalues are purely imaginary and its eigenvectors generally lie in the complexified Lie algebra. This leads us to consider complex-valued solutions of~\eqref{euler-arnold-coriolis}. The following theorem shows that, under a simple algebraic condition, the real part of such a solution again yields a genuine real solution.

\begin{theorem}\label{complex EAC to real EA}
    Suppose that for some $T>0$ there exist $X\in\g$ and a curve $\cecs \colon [0,T)\to \C\otimes\g$ in the complexification of the Lie algebra solving \eqref{euler-arnold-coriolis}. If $\cecs(t)=V(t)+iW(t)$ with $V(t),W(t)\in\g$ such that
    \begin{equation}
        \adstar_{W}W = 0
    \end{equation}
    then $V \colon [0,T)\to \g$ is a real-valued solution of \eqref{euler-arnold-coriolis}, and hence
    \begin{equation}
        U(t) = X + \Ad_{\gamma_X(t)}V(t)
    \end{equation}
    solves the Euler-Arnold equations \eqref{euler-arnold}. Furthermore
    \begin{equation*}
        Y(t) = \Ad_{\gamma_X(t)} W(t)
    \end{equation*}
    solves the corresponding linearized equation~\eqref{linearized Euler-Arnold} along $U(t)$.
\end{theorem}

\begin{proof}
    Taking the real and imaginary parts of
    \begin{equation*}
        \partial_t\cecs = - \adstar_\cecs\cecs - \adstar_\cecs X
    \end{equation*}
    gives
    \begin{equation}\label{real and imaginary EAC split}
        \begin{split}
            \partial_tV &= -\adstar_VV + \adstar_WW -\adstar_VX,\\
            \partial_tW &= -\adstar_VW -\adstar_WV - \adstar_WX.
        \end{split}
    \end{equation}
    Hence, the condition $\adstar_WW=0$ forces that $V$ is a real solution to \eqref{euler-arnold-coriolis} and the first assertion of the theorem follows.

    For the second, the linearized Euler-Arnold equation\footnote{A derivation of this can be found in \cite{misiolek2010fredholm}, however we rederive this during our perturbative argument later on.} along a solution $U$ is given by
    \begin{equation}\label{linearized EA}
        \partial_t Y + \adstar_UY + \adstar_YU = 0.
    \end{equation}
    Replacing $U = X + \Ad_{\gamma_X}V$ and $Y = \Ad_{\gamma_X}W$ in the left-hand side above yields
    \begin{align*}
        &\partial_t \big(\Ad_{\gamma_X}W\big) + \adstar_{X+ \Ad_{\gamma_X}V}\Ad_{\gamma_X}W + \adstar_{\Ad_{\gamma_X}W}\big(X + \Ad_{\gamma_X}V\big) \\
        &\quad = \Ad_{\gamma_X}\partial_t W + \ad_X\Ad_{\gamma_X}W + \adstar_X\Ad_{\gamma_X}W \\
        &\quad \qquad + \adstar_{\Ad_{\gamma_X}V}\Ad_{\gamma_X}W + \adstar_{\Ad_{\gamma_X}W}X + \adstar_{\Ad_{\gamma_X}W}\Ad_{\gamma_X}V \\
        &\quad = \Ad_{\gamma_X} \big( \partial_t W + \adstar_VW + \adstar_WV + \adstar_WX\big)
    \end{align*}
    where again we have used the fact that $\adstar_X = - \ad_X$ and Lemma \ref{general Killing lemma}. In light of \eqref{real and imaginary EAC split}, we see that $Y=\Ad_{\gamma_X}W$ indeed satisfies \eqref{linearized EA} with initial condition $Y(0) = W(0)$.
\end{proof}

We now show how an eigenfield of the map $\ef \mapsto \adstar_{\ef}X$ can be converted into a real solution of the Euler–Arnold equation. This result is the main algebraic ingredient in the explicit constructions presented in Sections~\ref{2D} and~\ref{3D}.

\begin{theorem}\label{eigenvalue to euler-arnold solution}
    Let $X\in\g$ with $\adstar_X = -\ad_X$. If there exists ${\ef} \in \C\otimes\g$ and $\kev \in \R$ such that
    \begin{equation*}
        \adstar_{{\ef}} X = -i \kev {\ef}~, \quad \text{and} \quad \adstar_{{\ef}} {\ef} = 0,
    \end{equation*}
    then, for $\rho \geq 0$ and $\sigma \in \R$, the curve
    \begin{equation}
        \cecs(t) = \rho e^{i(\sigma + \kev t)}{\ef}
    \end{equation}
    defines a complex-valued solution of \eqref{euler-arnold-coriolis}. Moreover, if ${\ef} = v + iw$, with $v,w \in \g$ and
    \begin{equation*}
        \adstar_ww=0
    \end{equation*}
    then the curves
    \begin{equation*}
        U(t) = X + \rho \Ad_{\gamma_X(t)} \big(\cos(\sigma + \kev t)v - \sin(\sigma + \kev t)w\big)
    \end{equation*}
    and
    \begin{equation*}
        Y(t) = \rho \Ad_{\gamma_X(t)} \big(\sin(\sigma + \kev t)v + \cos(\sigma + \kev t)w\big)
    \end{equation*}
    are real solutions to the Euler-Arnold equation \eqref{euler-arnold} and its linearization \eqref{linearized Euler-Arnold} respectively, where again $\gamma_X$ denotes the flow of $X$.

    Lastly, $U$ is stationary if and only if
    \begin{equation*}
        \adstar_X{\ef} = i \kev {\ef}.
    \end{equation*}
\end{theorem}

\begin{proof}
    The fact that $\cecs(t)=\rho e^{i(\sigma + \kev t)}{\ef}$ is a complex-valued solution of \eqref{euler-arnold-coriolis} is immediate.
    Writing $\cecs(t) = V(t) + iW(t)$, a direct computation gives 
    \begin{equation*}
        V(t) = \rho\big(\cos(\sigma + \kev t)v - \sin(\sigma + \kev t)w\big), \quad W(t) = \rho \big(\sin(\sigma + \kev t)v + \cos(\sigma + \kev t)w\big).
    \end{equation*}
    The condition that $\adstar_{\ef}{\ef}=0$ tells us that
    \begin{equation}
        \begin{split}
            &\adstar_vv - \adstar_ww = 0 \\
            &\adstar_vw + \adstar_wv = 0
        \end{split}
    \end{equation}
    The additional assumption $\adstar_w w = 0$ therefore gives $\adstar_vv=0$. From this we have
    \begin{equation*}
        \adstar_{W(t)}W(t) = \rho^2 \Big( \sin^2(\sigma + \kev t)\adstar_vv + \sin(\sigma + \kev t)\cos(\sigma + \kev t) \big(\adstar_vw + \adstar_wv\big) + \cos^2(\sigma + \kev t)\adstar_ww \Big) =0.
    \end{equation*}
    which, in light of Theorem \ref{complex EAC to real EA} establishes the first assertion of the theorem.

    For the stationarity condition, recall from Theorem \ref{Euler and Euler-Coriolis correspondence}, that $U(t) = X + \Ad_{\gamma_X}V$ is stationary if and only if $\partial_tV = \adstar_XV$. We compute the difference directly
    \begin{align*}
        \partial_tV - \adstar_XV &= \rho \Big( -\kev \sin(\sigma+\kev t) v -\kev \cos(\sigma+\kev t) - \cos(\sigma+\kev t)\adstar_Xv + \sin(\sigma+\kev t)\adstar_Xw\Big) \\
        &= \rho \Big(-\cos(\sigma+\kev t) \big(\adstar_Xv + \kev w\big) + \sin(\sigma+\kev t)\big(\adstar_Xw - \kev v\big)\Big).
    \end{align*}
    We can the above vanishes if and only if
    \begin{equation*}
        \adstar_Xv = -\kev w, \quad \text{and} \quad \adstar_Xw = \kev v
    \end{equation*}
    which is equivalent to the condition
    \begin{equation*}
        \adstar_X{\ef} = i \kev {\ef}.
    \end{equation*}
\end{proof}

\enlargethispage{2\baselineskip}

\begin{remark}
    In order to see Theorem \ref{simultaneous eigenvalue to euler solution} as a special instance of Theorem \ref{eigenvalue to euler-arnold solution}, note that, in the context of ideal hydrodynamics,
    \begin{equation*}
        \adstar_vw = \inert^{-1}[v, \inert w] \quad \text{for } v,w \in \ff.
    \end{equation*}
    Hence, the condition $\adstar_{\ef}X = - i \kev {\ef}$ is equivalent to $K_X{\ef} = - i \kev {\ef}$, while $\inert {\ef} = \iev {\ef}$ implies both $\adstar_{\ef}{\ef} = 0$ and $\adstar_ww = 0$.
    The added assumption of $[X,{\ef}] = -\ad_X {\ef} = i \aev {\ef}$ integrates to $\Ad_{\gamma_X(t)}{\ef} = e^{-i\aev t}{\ef}$ and notably avoids the need for $X$ to be Killing, although all examples in Sections \ref{2D} and \ref{3D} fall into this category. Finally, the stationarity condition $\adstar_X {\ef} = i\kev {\ef}$ in Theorem~\ref{eigenvalue to euler-arnold solution} becomes $[X,\inert {\ef}] = i\kev \inert {\ef}$ which, under the inertia eigenvalue condition $\inert {\ef} = \iev {\ef}$, is equivalent to $[X,{\ef}] = i\kev {\ef}$. This recovers the condition $\kev=\aev$ appearing in Theorem~\ref{simultaneous eigenvalue to euler solution}.
\end{remark}

\begin{remark}\label{euler-arnold trivial solutions}
    Recall that the flow of $U$, the real solution of the Euler-Arnold equation \eqref{euler-arnold}, and the flow of $V$, the corresponding real solution to \eqref{euler-arnold-coriolis} are related by the equation
    \begin{equation*}
        \gamma_{U} = \gamma_X \circ \gamma_V.
    \end{equation*}
    We note now that it may occur that $U$ is a non-stationary solution of \eqref{euler-arnold}, while also having that $V$ is a stationary solution of \eqref{euler-arnold-coriolis}. In this situation $U$ can be interpreted as a trivial solution in a moving frame generated by $X$. However, from the above proof, we see that
    \begin{equation*}
        \partial_t V = 0 \iff \kev = 0. 
    \end{equation*}
    Accordingly, eigenfields $\ef$ with $\kev = 0$ produce solutions that are trivial in this moving-frame sense, whereas $\kev \neq 0$ yields solutions $U$ that are genuinely non-trivial.
\end{remark}

We now motivate the preceding construction via a perturbative expansion of the Euler–Arnold equation. Given $u_0\in\g$, let $\gamma(t)=\exp_e(tu_0)$ denote the corresponding geodesic, which exists on some time interval $[0,T)$. As before, we define the Eulerian velocity $u(t)\in\g$ by $\partial_t\gamma(t)=d_eR_{\gamma(t)}u(t)$. We perturb the initial data by $w_0\in T_{u_0}\g$ and consider the two-parameter family of geodesics
\begin{equation*}
\Gamma(t,s)=\exp_e\big(t(u_0+sw_0)\big).
\end{equation*}
Let $U(t,s)\in\g$ denote the associated Eulerian velocity, defined by
\begin{equation*}
    \partial_t\Gamma(t,s)=d_eR_{\Gamma(t,s)}U(t,s).
\end{equation*}
We assume formally that $U$ admits a power-series expansion in $s$,
\begin{equation}\label{expansion}
U(t,s)=\sum_{n=0}^\infty X_n(t)s^n,
\end{equation}
where each $X_n(t)$ is a time-dependent element of $\g$.

\begin{lemma}
The coefficients $\{X_n(t)\}$ in~\eqref{expansion} satisfy
\begin{equation}\label{X_n evolution}
\frac{dX_n}{dt}+\sum_{k=0}^n \adstar_{X_k}X_{n-k}=0,
\end{equation}
with initial conditions
\begin{equation*}
    X_0(t)=u(t), \qquad X_1(0)=w_0, \qquad X_n(0)=0 \ \text{for } n\geq2.
\end{equation*}
\end{lemma}

\begin{proof}
Substituting the expansion~\eqref{expansion} into the Euler--Arnold equation~\eqref{euler-arnold} and using the linearity of $\adstar$ yields
\begin{equation*}
    \sum_{n=0}^\infty \frac{dX_n}{dt}s^n + \sum_{i=0}^\infty\sum_{j=0}^\infty \adstar_{X_i}X_j\, s^{i+j}=0.
\end{equation*}
Reindexing the second sum gives
\begin{equation*}
    \sum_{n=0}^\infty \frac{dX_n}{dt}s^n + \sum_{n=0}^\infty\sum_{k=0}^n \adstar_{X_k}X_{n-k}\, s^n=0.
\end{equation*}
Equating coefficients of $s^n$ yields~\eqref{X_n evolution}. The initial conditions follow from
\begin{equation*}
    U(t,0)=u(t), \qquad U(0,s)=u_0+sw_0.
\end{equation*}
\end{proof}

For $n=1$, equation~\eqref{X_n evolution} becomes
\begin{equation}\label{linearized Euler-Arnold}
    \frac{dX_1}{dt}+\adstar_{X_0}X_1+\adstar_{X_1}X_0=0,
\end{equation}
which is precisely the linearized Euler--Arnold equation along $u(t)=X_0(t)$. As observed in~\cite{ebin2006singularities}, this equation simplifies under left translation. Accordingly, for $n\geq0$ we define
\begin{equation}\label{left translation}
X_n(t)=\Ad_{\gamma(t)}x_n(t).
\end{equation}

\begin{lemma}
The functions $\{x_n(t)\}$ defined by~\eqref{left translation} satisfy
\begin{equation}\label{x_n evolution}
\begin{split}
&\Lambda(t)x_0(t)=u_0, \\
&\frac{d}{dt}\bigl(\Lambda(t)x_1(t)\bigr)+\adstar_{x_1(t)}u_0=0, \\
&\frac{d}{dt}\bigl(\Lambda(t)x_n(t)\bigr)
+\adstar_{x_n(t)}u_0
+\sum_{k=1}^{n-1}\adstar_{x_k(t)}\Lambda(t)x_{n-k}(t)=0,
\quad n\geq2,
\end{split}
\end{equation}
where $\Lambda(t):=\Adstar_{\gamma(t)}\Ad_{\gamma(t)}$, with initial conditions
\begin{equation*}
    \Ad_{\gamma(t)}x_0(t)=u(t), \qquad x_1(0)=w_0, \qquad x_n(0)=0 \ \text{for } n\geq2.
\end{equation*}
\end{lemma}

\begin{proof}
Applying $\Adstar_{\gamma(t)}$ to~\eqref{X_n evolution} for $n\geq2$ gives
\begin{equation*}
    \Adstar_{\gamma}\frac{dX_n}{dt} + \Adstar_{\gamma}\adstar_u X_n + \Adstar_{\gamma}\adstar_{X_n}u + \sum_{k=1}^{n-1}\Adstar_{\gamma}\adstar_{X_k}X_{n-k}=0.
\end{equation*}
Using~\eqref{Adstar derivative}, this becomes
\begin{equation}\label{x_n intermediate}
\frac{d}{dt}\bigl(\Adstar_{\gamma}X_n\bigr)
+\Adstar_{\gamma}\adstar_{X_n}u
+\sum_{k=1}^{n-1}\Adstar_{\gamma}\adstar_{X_k}X_{n-k}=0.
\end{equation}
Substituting $X_n=\Ad_{\gamma}x_n$ into~\eqref{x_n intermediate} and applying~\eqref{Adstar homomorphism} yields
\begin{equation*}
    \frac{d}{dt}\bigl(\Adstar_{\gamma}\Ad_{\gamma}x_n\bigr) + \adstar_{x_n}\Adstar_{\gamma}u + \sum_{k=1}^{n-1}\adstar_{x_k}\Adstar_{\gamma}\Ad_{\gamma}x_{n-k}=0.
\end{equation*}
Finally, invoking the conservation law~\eqref{general coadjoint conservation law} gives~\eqref{x_n evolution}.
\end{proof}

Recall now that for $X$ Killing we have $\Lambda_X(t)=\mathrm{Id}$, so \eqref{x_n evolution} reduces to
\begin{equation}\label{Killing x_n evolution}
\begin{split}
    &x_0 = X, \\
    &\frac{dx_1}{dt} + \adstar_{x_1}X=0, \\
    &\frac{d x_2}{dt} + \adstar_{x_2}X + \adstar_{x_1}x_{1}=0,\\
    &\frac{d x_n}{dt} + \adstar_{x_n}X + \sum_{k=1}^{n-1}\adstar_{x_k}x_{n-k}=0, \quad n\geq3
\end{split}
\end{equation}
with initial conditions
\begin{equation*}
    \qquad x_1(0)=w_0, \qquad x_n(0)=0 \ \text{for } n\geq2.
\end{equation*}

The simplest nontrivial scenario occurs when the hierarchy truncates after the first-order term, i.e., when $x_n\equiv 0$ for all $n\geq2$. Remarkably, a sufficient condition for this is $\adstar_{x_1}x_1=0$; indeed, then $x_2=0$ and all higher-order terms vanish inductively. In this case, the system reduces to
\begin{equation}\label{Killing x_1 evolution}
\begin{split}
    &x_0 = X, \\
    &\frac{dx_1}{dt} + \adstar_{x_1}X=0, \\
    &\adstar_{x_1}x_{1}=0,\\
\end{split}
\end{equation}
with initial condition
\begin{equation*}
    \qquad x_1(0)=w_0.
\end{equation*}
In this situation, the corresponding Euler–Arnold solution takes the form
\begin{equation*}
    U(t)=X+\Ad_{\gamma_X(t)}x_1(t),
\end{equation*}
which is precisely the structure identified in Theorem~\ref{Euler and Euler-Coriolis correspondence}.

\enlargethispage{\baselineskip}

\section{Future Directions}\label{future}

There are several natural directions for continuations of this work. The first concerns the rigidity of the mechanism underlying our constructions. In all of our explicit examples, the steady Euler flow $u_0$ and its image under the inertia operator $\inert u_0$ are both Killing fields. However, it is not clear whether this is essential. It would be interesting to determine whether this is the only mechanism by which the hypotheses of Theorem~\ref{simultaneous eigenvalue to euler solution} can be satisfied, or whether new classes of solutions arise when these assumptions are relaxed.

Closely related to this is the question of which three-dimensional Riemannian manifolds admit Killing fields whose curl is also Killing. A complete classification would provide a clearer geometric understanding of the scope of the constructions presented here and may reveal additional families of explicit Euler flows. We note that complete Killing fields on two-dimensional manifolds (possibly with boundary, possibly noncompact) have been recently classified by Shimizu \cite{shimizu2024hydrodynamic}, who obtained a full description of all topological possibilities. A full classification in the three-dimensional case would include circle bundles over manifolds with boundary as opposed to the boundaryless case considered in Section \ref{circle}, along with the classification of all possible bundles over a given surface of positive genus in Theorem \ref{boothbywangprop}. The nonorientable case may also be interesting to study as in \cite{khesin2025curvatures}. For torus bundles as in Section \ref{torus}, we considered only trivial bundles over surfaces with boundary, but we do not have a global picture or know if there are any examples without boundary.

We know that even steady flows can have chaotic trajectories, such as the ABC flows~\cite{dombre1986chaotic} on a $3$-torus. Our solutions are linear time-dependent combinations of steady flows, which in some cases have closed trajectories. The sinusoidal dependence can create precession and lead to trajectories that fill regions of space. It would be interesting to see if there can be any closed particle trajectories, or if it is possible to have all trajectories close at the same time (corresponding to closed geodesics in the volume-preserving diffeomorphism group). 

As noted in Theorem \ref{simultaneous eigenvalue to euler solution}, the construction yields not only an exact solution of the Euler equations, but also a particular exact solution of the linearized Euler equation around it, which oscillates in time. However we do not know any other solutions of the linearized equation, and it is far from clear whether these solutions are stable in the Eulerian sense. Such questions are difficult even for the Rossby-Haurwitz waves on $S^2$, and nonlinear stability presents greater challenges~\cite{constantin2022stratospheric}; however it is conceivable some simplifications may happen in highly-symmetric spaces as on $S^3$.

Related to the question of Eulerian stability is the issue of Lagrangian stability, or deviation of particle paths. The geometric approach is to study growth and boundedness of Jacobi fields along the corresponding geodesic in the volumorphism group. Similarly one can study conjugate points, where the geodesics meet to first-order. This project was initiated by Benn~\cite{benn2021conjugate} for Rossby-Haurwitz waves on $S^2$, who found conjugate points along most of the corresponding geodesics. It seems likely that on $S^3$ for example there will also be conjugate points along most of the corresponding geodesics, especially considering that they are much easier to find in three dimensions than two~\cite{preston2006on}. 

Since all our solutions are trigonometric in time, one could view the solutions as an explicit embedding of harmonic oscillators into non-steady Euler flows. This is analogous to the fact that any finite-dimensional dynamical system can be embedded into a steady Euler flow in sufficiently high dimension~\cite{tao2018universality} and sometimes even in three dimensions~\cite{cardona2021constructing}. Can we get more complicated time dependence of solutions (not just sinusoidal combinations of three steady flows) by searching in higher-dimensional manifolds? Our inertia operator approach is specific to two and three dimensions, so we do not have any information about whether such constructions extend to higher dimensions. 

Beyond ideal hydrodynamics, it is natural to ask whether similar ideas can be applied to related nonlinear systems. In particular, the equations of ideal magnetohydrodynamics in three dimensions admit a formulation closely parallel to that of the Euler equations, with the vorticity equation coupled to the evolution of a divergence-free magnetic field. This structure suggests that a perturbative or geometric approach analogous to the one developed here may yield explicit, non-stationary solutions of the magnetohydrodynamics equations as well, as happens for Kelvin modes~\cite{dritschel1991generalized}. 

One might expect the same method to give explicit solutions of the Navier-Stokes equation on a manifold in some cases. Here of course the boundary conditions change, so it would likely be easier to find analogues on compact manifolds without boundary. Here one should take care to use the correct vector Laplacian~\cite{chan2017formulation} on the Riemannian $3$-manifold. On an Einstein $3$-manifold such as $S^3$, this problem may be tractable.

\appendix

\section{Auxiliary Geometric Results}\label{geometric lemmas}

Here we present the proofs of Proposition \ref{abundance lemma}, Proposition \ref{beltramilemma} and Theorem \ref{cmetriccurls}.

\subsection{The proof of Proposition \ref{abundance lemma}}

We begin by establishing a useful symmetry.

\begin{lemma}\label{antisymmetry}
    Let $(M,g)$ be a two- or three-dimensional compact Riemannian manifold, possibly with boundary. For all $u,v,w\in \ff(M)$ we have 
    \begin{equation}\label{biinvariant}
        \langle \inert^{-1} u, [v,w] \rangle_{L^2} + \langle \inert^{-1}w, [v,u]\rangle_{L^2} = 0.
    \end{equation}
\end{lemma}

\begin{proof}
    For $v \in \ff$ consider the symmetric bilinear operator given by
    \begin{equation*}
        \Psi_v : \ff \times \ff \rightarrow \R , \quad (u,w) \mapsto \ip{ \inert^{-1} u, [v,w]}_{L^2} + \ip{\inert^{-1}w, [v,u]}_{L^2}.
    \end{equation*}
    We will show that $\Psi_v \equiv 0$. By polarization it suffices to show that $\Psi_v$ vanishes on the diagonal of $\ff \times \ff$.

    If $M$ is two-dimensional, we can express any $u \in \ff$ as a skew-gradient $u = \symp \psi_u$ for some stream function $\psi_u$ vanishing on the boundary $\partial M$. Moreover, the commutator of vector fields can be rewritten in terms of the Poisson bracket as $[v,u] = \symp \{\psi_v, \psi_u\}$ which is defined in terms of the Riemannian volume form $\mu$ by $\{\psi_v, \psi_u\}\mu = d\psi_v \wedge d\psi_u$. Using this we have
    \begin{align*}
        \frac{1}{2}\Psi_v(u,u) &= \ip{ \Delta^{-1} u, [v,u]}_{L^2} \\
        &= \ip{ \Delta^{-1} \symp \psi_u , \symp\{\psi_v,\psi_u\}}_{L^2} \\
        &= \int_M g\left( \symp \Delta^{-1} \psi_u, \symp \{\psi_v, \psi_u \} \right) \mu.
    \end{align*}
    As the rotation is an isometry, we may replace the skew-gradients with gradients and we have
    \begin{align*}
        \frac{1}{2}\Psi_v(u,u) &= \int_M g\left( \nabla \Delta^{-1} \psi_u, \nabla \{\psi_v, \psi_u \} \right) \mu \\
        &= \int_M \Big( \diver \big(\Delta^{-1}\psi_u \nabla \{\psi_v, \psi_u\}\big) + \psi_u\{\psi_v, \psi_u\} \Big) \mu.
    \end{align*}
    For the divergence term, note that
    \begin{align*}
        \int_M \diver \big(\Delta^{-1}\psi_u \nabla \{\psi_v, \psi_u\}\big) \mu &= \int_{\partial M} \Delta^{-1}\psi_u \nabla \{\psi_v, \psi_u\} \iota_\nu \mu
    \end{align*}
    while for the Poisson bracket term we have
    \begin{align*}
        \int_M \psi_u\{\psi_v, \psi_u\} \mu & = \int_M \psi_u~d\psi_v \wedge d\psi_u = - \frac{1}{2}\int_M d(\psi_u^2 d \psi_v) = - \frac{1}{2}\int_{\partial M} \psi_u^2 d\psi_v.
    \end{align*}
    Both of these vanish on account of the boundary conditions for the stream functions.

    For the three-dimensional case we have $[v,u] = \curl (u \times v)$, where $\times$ denotes the usual cross product. Hence, as $\curl^{-1}$ is $L^2$ self-adjoint, we have
    \begin{equation*}
        \frac{1}{2}\Psi_v(u,u) = \ip{\curl^{-1} u , [v,u]}_{L^2} = \ip{u, u \times v}_{L^2} = 0.
    \end{equation*}
\end{proof}

As an immediate consequence, we obtain the following property of the coadjoint operator, $K_{u_0}$.

\begin{corollary}\label{coadjoint antisymmetry}
    Let $(M,g)$ be a two- or three-dimensional compact Riemannian manifold, possibly with boundary. For any $u_0 \in \ff(M)$, the operator $K_{u_0} : \ff(M) \rightarrow \ff(M) \ ; \ v \mapsto \inert^{-1}[v, \inert u_0]$ is $L^2$ skew self-adjoint.
\end{corollary}

We are now ready to construct the basis of complex fields $\ef$ satisfying the hypothesis of Theorem \ref{simultaneous eigenvalue to euler solution}.

\begin{proposition}[= Proposition \ref{abundance lemma}]
    Let $(M,g)$ be a Riemannian manifold, possibly with boundary, and $X \in \ff(M)$ be a Killing field such that $\inert X$ is also Killing. Then the operators
    \begin{equation*}
        K_X : \ef \mapsto K_X\ef, \quad \inert : \ef \mapsto \inert \ef, \quad \ad_X: \ef\mapsto -[X,\ef]    
    \end{equation*}
    admit a discrete simultaneous eigenbasis of $\C \otimes\ff(M)$. 
\end{proposition}

\begin{proof}[Proof of Proposition~\ref{abundance lemma}]
    From Lemma \ref{inertia basis} we have the $L^2$-orthogonal decomposition
    \begin{equation*}
        \C \otimes \ff(M) = \bigoplus_{k=1}^\infty E_k
    \end{equation*}
    where, for each $k \in \Z_{\geq1}$, we have $\dim(E_k) < \infty$ and $\inert \ef = \iev_k \ef$, for all $\ef \in E_k$.

    We show that, on each $E_k$, the operators $K_X$ and $\ad_X$ are skew-adjoint and commute with $\inert$ and with each other. It then follows that each $E_k$ admits a simultaneous eigenbasis, yielding the desired decomposition.

    We begin with the commutator relations. Note first that $K_X = \inert^{-1}\ad_{\inert X}$ and that, as both the flows of $X$ and $\inert X$ act by isometries, we have $[\ad_X, \inert] = [\ad_{\inert X}, \inert] = 0$. Next, as the adjoint is a homomorphism and $X$ is a steady-state solution of \eqref{inertia euler}, we have that $[\ad_X, \ad_{\inert X}] = \ad_{[X, \inert X]} = 0$. Hence it follows that $[K_X, \ad_X] = [K_X, \inert]= 0$.

    The skew-symmetry of $K_X$ has already been established in Corollary \ref{coadjoint antisymmetry}. For $\ad_X$ note that, for all $\ef \in E_k$, we have
    \begin{equation*}
       \int_M g\big( \ad_X\ef, \ef\big) \mu  =  - \int_M g\big( [X,\ef], \ef\big) \mu =\int_M g(\nabla_{\ef}X,{\ef}) -   g(\nabla_X\ef, \ef) \, \mu
    \end{equation*}
    The first term vanishes as $X$ is Killing, cf. \eqref{Killing condition}. For the second term we integrate by parts and acquire
    \begin{align*}
        \int_M g\big( \ad_X \ef, \ef\big) \mu &= -\tfrac{1}{2} \int_M X g\big({\ef},{\ef}) \mu \\
        &= -\tfrac{1}{2} \int_{\partial M} g({\ef},{\ef}) g(X,\nu) \, \iota_{\nu}\mu + \tfrac{1}{2} \int_M (\diver{X}) g({\ef},{\ef}) \, \mu,
    \end{align*}
    where again $\nu$ denotes the normal vector to the boundary and $\mu$ the volume-form. Both terms vanish pointwise as $X$ is divergence-free and tangent to the boundary. Hence, as in Lemma \ref{antisymmetry}, we have that the symmetric bilinear operator
    \begin{equation*}
        \Phi_X : E_k \times E_k \rightarrow \R, \quad (\ef_1, \ef_2) \mapsto \ip{\ef_1, \ad_X \ef_2}_{L^2} + \ip{\ef_2, \ad_X \ef_1}_{L^2} 
    \end{equation*}
    vanishes identically and the result follows.
\end{proof}

\subsection{The proof of Proposition \ref{beltramilemma}}

We first recall the statement.

\begin{proposition}[= Proposition \ref{beltramilemma}]
    Let $(M,g)$ be a three-dimensional compact connected Riemannian manifold, possibly with boundary. If $X$ and $\curl{X}$ are both Killing fields, then $g(X,\curl X)$ is constant. Consequently, if $\curl X = fX$, then both $\abs{X}_g$ and $f$ are constant.
\end{proposition}

\begin{proof}[Proof of Proposition~\ref{beltramilemma}]
    Let $\vf(M)$ denote the space of smooth vector fields on $M$ which are tangent to the boundary and $\nabla$ the Levi-Civita connection for $g$. The conditions that $X$ and $\curl X$ are Killing can be expressed as
    \begin{equation}\label{X Killing}
        g(\nabla_v X, w) + g(v, \nabla_w X) = 0
    \end{equation}
    and
    \begin{equation}\label{curl X Killing}
        g(\nabla_v \curl X, w) + g(v, \nabla_w \curl X) = 0
    \end{equation}
    for all $v, w \in \vf(M)$ respectively. By the definition of the cross product, for all $w \in \vf(M)$, we have
    \begin{equation}\label{cross product condition}
        g(\nabla_{\curl X}X, w) - g(\curl X , \nabla_w X) = g(\curl X , \curl X \times w) = 0.
    \end{equation}
    Combining this with \eqref{X Killing} with $v = \curl X$ immediately gives
    \begin{equation*}
        \nabla_{\curl X} X = 0
    \end{equation*}
    which, by \eqref{cross product condition} yields
    \begin{equation}\label{miscellaneous Killing identity 1}
        g(\curl X, \nabla_w X) = 0
    \end{equation}
    for all $w \in \vf(M)$.
    
    Next, note that\footnote{See \cite{lichtenfelz2022axisymmetric} for example.}
     \begin{equation}\label{killingcross}
        X\times \curl{X} = \nabla \abs{X}_g^2.
    \end{equation}
    In particular $X$ and $\curl X$ are linearly independent on an open set if and only if $\abs{X}_g$ is non-constant. Taking the curl of the left hand side of \eqref{killingcross} gives
    \begin{align*}
        \curl \big( X \times \curl X\big) &= \nabla_{\curl X} X - \nabla_X \curl X + \diver(\curl X) X - \diver(X) \curl X \\
        &= \nabla_{\curl X} X - \nabla_X \curl X.
    \end{align*}
    Hence, as $\curl \nabla \abs{X}_g^2 = 0$, we have
    \begin{equation}\label{miscellaneous Killing identity 2}
        \nabla_X\curl X=\nabla_{\curl X} X = 0
    \end{equation}
    and, in particular\footnote{As mentioned earlier, this implies that any Killing field is automatically a steady-state solution of \eqref{inertia euler}.}
    \begin{equation}\label{Killing Lie bracket}
        [X, \curl X] = 0.
    \end{equation}
    Now, letting $w = X$ in \eqref{curl X Killing} and using \eqref{miscellaneous Killing identity 1}, \eqref{miscellaneous Killing identity 2} gives that
    \begin{equation*}
        v \big(g(X, \curl X)\big) = g(\nabla_v X, \curl X) + g(X, \nabla_v \curl X) = g(\nabla_v X, \curl X) - g(\nabla_X \curl X, v) = 0
    \end{equation*}
    for any $v \in \vf(M)$.
    Hence $g( X,\curl X)$ is constant on $M$.
    
    Lastly, if $\curl X = fX$ for some function $f$, then by \eqref{killingcross}, we must have that $\abs{X}_g$ is constant, and hence $f=\frac{g( X,\curl X)}{\abs{X}_g^2}$ is also constant. 
\end{proof}

\subsection{The proof of Theorem \ref{cmetriccurls}}

We first recall the statement.

\begin{theorem}[= Theorem \ref{cmetriccurls}]
    Consider the three-dimensional manifold $[\radiusone,\radiustwo]\times \mathbb{T}^2$ with metric \eqref{3disometriesmetric}, which again, for $c>0$ and $\fnone$ a positive function, is given by
\begin{equation*}
    ds^2 = dr^2 + \Big(\fnone(r) d\theta + \frac{c}{\fnone(r)}dz\Big)^2 + \fnone'(r)^2dz^2.
\end{equation*}
For $m,n \in \Z$, let $\fnthree$ and $\fnfour$ solve the Sturm-Liouville type system
\begin{align*}
\alpha \fnone(r) \fnone'(r) \fnthree'(r) &= n\left( m - \frac{cn}{\fnone(r)^2}\right) \, \fnthree(r) + \big( \alpha^2 \fnone(r)^2 - n^2\big) \fnfour(r) \\
\alpha \fnone(r)\fnone'(r) \fnfour'(r) &= \left( 2c\alpha \, \frac{\fnone'(r)^2}{\fnone(r)^2} + \left( m-\frac{cn}{\fnone(r)^2}\right)^2 - \alpha^2 \fnone'(r)^2 \right) \, \fnthree(r) - n\left( m - \frac{cn}{\fnone(r)^2}\right) \, \fnfour(r),
\end{align*}
with boundary conditions 
\begin{equation*}
n\fnfour(\radiusone) - \left( m- \frac{cn}{\fnone(\radiusone)^2}\right) \fnthree(\radiusone) = n\fnfour(\radiustwo) - \left( m- \frac{cn}{\fnone(\radiustwo)^2}\right) \fnthree(\radiustwo) = 0,
\end{equation*}
the constant $\iev$ implicitly defined and $\fnfive$ given by
\begin{equation*}
    \fnfive(r) = \,\frac{n\fnfour(r)-\big(m-\frac{cn}{\fnone(r)^2}\big)\fnthree(r)}{\alpha \fnone(r)\fnone'(r)}.
\end{equation*}
Then the complex fields
\begin{equation*}
    \ef(r,\theta,z) = e^{in\theta} e^{imz} \bigg( i \fnfive(r) \, \partial_r + 
    \frac{\fnthree(r)}{\fnone(r)^2} \, \partial_{\theta} + \frac{\fnfour(r)}{\fnone'(r)^2} \Big( -\frac{c}{\fnone(r)^2} \, \partial_{\theta} + \partial_z\Big)
    \bigg)
\end{equation*}
are curl-eigenfields satisfying the conditions of Theorem \ref{simultaneous eigenvalue to euler solution} with
\begin{equation*}
        K_{\partial_\theta} \ef = -i \frac{2m}{\alpha} \ef, \qquad \curl \ef = \iev \ef, \qquad [\partial_\theta, \ef] = i n \ef.
    \end{equation*}
\end{theorem}

\begin{proof}[Proof of Theorem~\ref{cmetriccurls}]
    From the explicit form of the metric, we see that an orthogonal basis of $1$-forms is given by\footnote{Here we use a non-orthonormal basis of $1$-forms, unlike in Theorem \ref{transversecurltheorem}, since the $d\sigma^k$ are simpler with this choice.} 
\begin{equation}\label{cmetric1forms}
    \sigma^1 = dr, \qquad \sigma^2 = d\theta + \frac{c}{\fnone(r)^2}\, dz, \qquad \sigma^3 = dz,
    \end{equation}
    satisfying 
    $$ \lvert \sigma^1\rvert_g = 1, \qquad \lvert \sigma^2\rvert_g = \frac{1}{\fnone(r)}, \qquad \lvert \sigma^3\rvert_g = \frac{1}{\fnone'(r)}.$$
    Applying the Hodge star and the exterior derivative to each yields
    \begin{equation*}
        \star \sigma^1 = \fnone(r)\fnone'(r) \, \sigma^2\wedge \sigma^3 , \quad \star \sigma^2 = \frac{\fnone'(r)}{\fnone(r)}\, \sigma^3\wedge \sigma^1, \quad \star \sigma^3 = \frac{\fnone(r)}{\fnone'(r)}\,\sigma^1\wedge \sigma^2.
    \end{equation*}
    and
    \begin{equation*}
        d\sigma^1 = 0, \qquad d\sigma^2 = \frac{2c\fnone'(r) }{\fnone(r)^3}\, \sigma^3\wedge \sigma^1, \qquad d\sigma^3 = 0
    \end{equation*}
    respectively. A curl eigenfield will satisfy $\curl{v} = \alpha v$ for some $\alpha\in\mathbb{R}$, which is equivalent to writing 
$$ dv^{\flat} = \alpha \star\! v^{\flat}.$$
Writing $v$ in Fourier components, we have (since the orbits of $\partial_{\theta}$ and $\partial_z$ both have length $2\pi$)
\begin{equation}\label{Wflatdef} 
v^{\flat} = e^{in\theta} e^{imz} 
\left(i \fnfive(r) \, \sigma^1 + \fnthree(r) \, \sigma^2 + \fnfour(r)\, \sigma^3 \right),
\end{equation}
for integers $m$ and $n$ and some purely radial functions $f,g,h$.
We compute that 
\begin{align*}
    dv^{\flat} &= e^{in\theta}e^{imz} \big( in \sigma^2 + i (m-\tfrac{cn}{\fnone(r)^2}) \sigma^3\big)\wedge \left(i \fnfive(r) \, \sigma^1 + \fnthree(r) \, \sigma^2 + \fnfour(r)\, \sigma^3 \right) \\
    &\qquad\qquad + e^{in\theta}e^{imz} \big( \fnthree'(r) \, \sigma^1\wedge \sigma^2 - \fnfour'(r) \sigma^3\wedge \sigma^1\big) 
    + e^{in\theta}e^{imz} \, \frac{2c\fnone'(r) \fnthree(r)}{\fnone(r)^3} \, \sigma^3\wedge \sigma^1 \\
    &= e^{in\theta}e^{imz} \bigg( 
    \big( \fnthree'(r) + n \fnfive(r)\big) \sigma^1\wedge \sigma^2 
    + i \Big( n\fnfour(r) - (m-\tfrac{cn}{\fnone(r)^2}) \fnthree(r) \Big) \sigma^2\wedge \sigma^3 + \\
    &\qquad\qquad + \Big( -\fnfour'(r) + \frac{2c\fnone'(r)\fnthree(r)}{\fnone(r)^3} - \left( m-\frac{cn}{\fnone(r)^2}\right) \fnfive(r) \Big) \sigma^3\wedge \sigma^1 \bigg)  
\end{align*}
while the Hodge star formulas
lead to 
\begin{align*}
    \star v^{\flat} &=  
    e^{in\theta}e^{imz} \left( i\fnone(r) \fnone'(r) \fnfive(r) \,\sigma^2\wedge \sigma^3 + \frac{\fnone'(r)}{\fnone(r)} \, \fnthree(r) \, \sigma^3\wedge \sigma^1+ \frac{\fnone(r)}{\fnone'(r)}\, \fnfour(r) \, \sigma^1\wedge \sigma^2\right)
\end{align*}

Matching components then gives 
\begin{align*}
     n\fnfour(r) - \left(m-\frac{cn}{\fnone(r)^2}\right) \fnthree(r) &= \alpha \fnone(r)\fnone'(r) \fnfive(r) \\
     \fnthree'(r) + n \fnfive(r) &= \alpha\,\frac{\fnone(r)}{\fnone'(r)}\, \fnfour(r) \\
     \fnfour'(r) - \frac{2c\fnone'(r)\fnthree(r)}{\fnone(r)^3} + \left( m-\frac{cn}{\fnone(r)^2}\right) \fnfive(r) &= -\alpha\, \frac{\fnone'(r)}{\fnone(r)} \, \fnthree(r).
\end{align*}
Using the first equation to solve for $\fnfive$ (assuming $\alpha$ is nonzero), we get the system \eqref{geqcmetric}--\eqref{heqcmetric}. Since the normal derivative at the boundary tori is $\partial_r$, the condition that $v$ is tangent to the boundary is equivalent to $\fnfive(a)=\fnfive(b)=0$, which translates into \eqref{cmetricbcs}. 

The only thing remaining is to lift $v^{\flat}$ given by \eqref{Wflatdef} to a vector field. From the metric \eqref{3disometriesmetric} we have that 
$$ (\sigma^1)^{\sharp} = \partial_r, \qquad (\sigma^2)^{\sharp} = \frac{1}{\fnone(r)^2} \, \partial_{\theta}, \qquad (\sigma^3)^{\sharp} = \frac{1}{\fnone'(r)^2} \left( \partial_z - \frac{c}{\fnone(r)^2} \, \partial_{\theta}\right).$$
The formula \eqref{cmetriccurlfield} then follows directly.
\end{proof}

\bibliographystyle{amsplain}
\bibliography{bibliography.bib}

\vfill

\end{document}